\documentclass[a4paper,11pt]{amsart}

\usepackage[utf8]{inputenc}
\usepackage[english]{babel}	
\usepackage[T1]{fontenc}

\usepackage{amsmath,amssymb,amsfonts,amsthm,graphicx,mathtools,url} 

\usepackage{tikz-cd}
\usepackage[all,cmtip]{xy}

\usepackage{a4wide}					%Allarga margini

\usepackage{}								%Serve per  := (\coloneqq)

\usepackage{color}

\usepackage{enumerate}

\usepackage{todonotes}      

\usepackage[hyperindex,breaklinks,pdftex, bookmarks,]{hyperref}

\usepackage[alphabetic, msc-links, nobysame]{amsrefs}

\mathchardef\mhyphen="2D

\numberwithin{equation}{section}

\theoremstyle{plain}
\newtheorem{thm}{Theorem}[section]
\newtheorem*{theointra}{Theorem A}
\newtheorem*{theointrb}{Theorem B}
\newtheorem{prop}[thm]{Proposition}

\newtheorem{lemma}[thm]{Lemma}

\theoremstyle{definition}
\newtheorem{ex}[thm]	{Example}
\newtheorem{df}[thm]{Definition}

\newtheorem{rem}[thm]{Remark}

\newtheorem{notation}[thm]	{Notation}

\newcommand{\x}{\times}
\renewcommand{\o}{\otimes}
\renewcommand{\phi}{\varphi}
\newcommand{\G}{G}
\newcommand{\id}{\mathrm{id}}

\DeclareMathOperator{\Hom}{Hom}
\DeclareMathOperator{\Mor}{Mor}

\DeclareMathOperator{\colim}{colim}

\newcommand{\N}{\mathbb{N}}
\newcommand{\Z}{\mathbb{Z}}

\newcommand{\bC}{\mathbf{C}}

\newcommand{\Gcan}{G_{\mathrm{can,min}}}

\DeclareMathOperator{\supp}{supp}
\DeclareMathOperator{\Tot}{Tot}
\DeclareMathOperator{\CN}{CN}
\DeclareMathOperator{\XHC}{\mathcal{X}HC}
\DeclareMathOperator{\XHH}{\mathcal{X}HH}
\DeclareMathOperator{\HC}{HC}
\DeclareMathOperator{\HH}{HH}
\DeclareMathOperator{\XH}{\mathcal{X}H}
\DeclareMathOperator{\XC}{\mathcal{X}\mathrm{Mix}}
\DeclareMathOperator{\Nerve}{N}

\DeclareMathOperator{\loc}{loc}

\newcommand{\A}{\mathcal{A}}
\newcommand{\bA}{\mathbf{A}}

\newcommand{\B}{\mathcal{B}}
\newcommand{\D}{\mathcal{D}}

\newcommand{\CC}{\mathcal{C}}
\renewcommand{\P}{\mathcal{P}}
\newcommand{\U}{U}
\newcommand{\cU}{\mathcal{U}}
\newcommand{\cO}{\mathcal{O}}

\newcommand{\X}{\mathcal{X}}
\newcommand{\Y}{\mathcal{Y}}

\newcommand{\tr}{\mathrm{tr}}

\newcommand{\Sp}{\mathbf{Sp}}

\newcommand{\cX}{\mathcal{X}}
\newcommand{\Mix}{\mathbf{Mix}}
\newcommand{\fMix}{\mathrm{Mix}}
\newcommand{\Mod}{\mathbf{Mod}}

\newcommand{\LMod}{\Lambda\textbf{-}\mathbf{Mod}}
\newcommand{\BC}{\mathbf{BornCoarse}}

\newcommand{\bCh}{\mathbf{Ch}}

\newcommand{\dgcat}{\mathbf{dgcat}}
\newcommand{\Cat}{\mathbf{Cat}}
\newcommand{\bdgcat}{\mathbf{dgcat}}
\newcommand{\Add}{\mathbf{Add}}
\newcommand{\Fin}{\mathbf{Fin}}

\usepackage{microtype}

\title{Cyclic homology for bornological coarse spaces}
\author{Luigi Caputi}
\address{Fakult\"at f\"ur Mathematik, Universit\"at Regensburg,
	Universit\"atsstra\ss{}e 31, 93040 Regensburg, Germany}
\email{luigi.caputi@ur.de}
\begin{document}

\begin{abstract}
	The goal of the paper is to define Hochschild and cyclic homology for bornological coarse spaces, \emph{i.e.,}  
	lax symmetric monoidal  functors $\XHH_{}^G$ and $\XHC_{}^G$ 
	from the category $G\BC$ of equivariant bornological coarse spaces  to the cocomplete stable $\infty$-category $\bCh_\infty$ of chain complexes reminiscent of the classical Hochschild and cyclic homology. 
We  investigate relations to  
coarse algebraic $K$-theory $\X K^G_{}$ and to  coarse ordinary homology~$\XH^G$ by  constructing a trace-like natural transformation $\X K_{}^G\to  \XH^G$  that factors through  coarse Hochschild (and cyclic) homology.
 We further  compare  the forget-control map for $\XHH_{}^G$ with the associated generalized assembly map.
\end{abstract}

\maketitle

\setcounter{tocdepth}{1}
\tableofcontents

\section*{Introduction}

Coarse geometry is the study of metric spaces from a large-scale point of view 
\cite{roe2,roe3, roe, mitchener}.
A new axiomatic and homotopic approach to coarse geometry and coarse homotopy theory has been recently developed  by Bunke and Engel \cite{bunke}. 
In this set-up, the main objects are called \emph{bornological coarse spaces} \cite[Def.~2.5]{bunke}, and every metric space is a bornological coarse space in a canonical way. In the equivariant setting,  if $G$ is a group,
$G$-bornological coarse spaces are bornological coarse spaces with a $G$-action by automorphisms \cite[Def.~2.1]{bunkeeq}. Among various invariants of $G$-bornological coarse spaces we are interested in  \emph{equivariant coarse homology theories}, \emph{i.e.,} functors
\[
E\colon G\BC\to \bC
\]
from the category $G\BC$ of $G$-bornological coarse spaces  to a cocomplete stable $\infty$-category $\bC$, satisfying some additional axioms: coarse invariance, flasqueness, coarse excision and u-continuity \cite[Def.\,3.10]{bunkeeq}. Examples of coarse homology theories arise as coarsifications of locally finite  homology theories \cite{bunke}. Among  other theories, there are coarse versions of ordinary homology and of topological $K$-theory \cite{bunke}, of equivariant algebraic  $K$-homology and of topological Hochschild homology \cite{bunkeeq,bunkecisinski}, and of Waldhausen's $A$-theory~\cite{2018arXiv181111864B}.

Classically, 
Hochschild and cyclic homologies have  been defined as homology invariants of algebras \cite{loday}, then extended to invariants of dg-algebras, schemes, additive categories and exact categories~\cite{keller,mccarthy}. 
The aim of the paper is twofold: we construct  coarse homology theories
defining Hochschild and cyclic homology 
for   bornological coarse spaces and then we study their relations to coarse algebraic $K$-theory and coarse ordinary homology. We remark that these coarse homology theories can be abstractly defined  by using a universal  equivariant coarse homology theory constructed by Bunke and Cisinski  \cite{bunkecisinski}. However, we choose to provide a more concrete construction with the hope that it might be more suitable for computations (see, \emph{e.g.,} the application to the construction of the natural transformation to coarse ordinary homology, Theorem \ref{suzfagsazgsafg}). We now describe the main results of the paper.

\vline

Let $k$ be field and let $G$ be a group. 
We denote by $C_{*}^{\mathrm{HH}}$ and $C_{*}^{\mathrm{HC}}$ the chain complexes computing  Hochschild homology and cyclic homology (of $k$-algebras) respectively. The $G$-bornological coarse space $G_{\mathrm{can,min}}$ denotes a canonical  bornological coarse space associated to the group $G$ (see Example~\ref{equivex}). Let $\bCh_\infty$ be the $\infty$-category of chain complexes. 
The following is a combination of  Theorem \ref{gajhljaghjlgalaöghga}, Proposition \ref{XCCsymmmon}, Proposition~\ref{pointHC} and Proposition~\ref{iso}:

\begin{theointra}
	There are lax symmetric monoidal  functors 
	\[
	\mathcal{X}\mathrm{HH}_{k}^G\colon G\mathbf{ BornCoarse}\to \mathbf{Ch}_\infty \quad\text{ and }\quad \mathcal{X}\mathrm{HC}_{k}^G\colon G\mathbf{ BornCoarse}\to \mathbf{Ch}_\infty
	\] 
	satisfying the following properties:
	\begin{enumerate}[(i)]
		\item 	$\mathcal{X}\mathrm{HH}_{k}^G$ and $\mathcal{X}\mathrm{HC}_{k}^G$ are $G$-equivariant coarse homology theories;
		\item there are  equivalences of chain complexes  \[\mathcal{X}\mathrm{HH}_{k}^G(*)\simeq C_{*}^{\mathrm{HH}}(k) \quad  
		\text{ and } \quad 
		\mathcal{X}\mathrm{HC}_{k}^G(*)\simeq C_{*}^{\mathrm{HC}}(k)\]  
		between the evaluations of $\XHH_k^G$ and $\XHC_k^G$ at the one-point bornological coarse space $\{*\}$, endowed with the trivial $G$-action,  and the chain complexes computing Hochschild and cyclic homology of $k$;
		\item there are equivalences 
		\[
		\XHH_k^\G(\G_{\mathrm{can,min}})\simeq C_{*}^{\HH}(k[\G];k)
		\quad		\text{ and }\quad
		\XHC_k^\G(\G_{\mathrm{can,min}})\simeq C_{*}^{\HC}(k[\G];k)\]
		of chain complexes between the evaluations at the $G$-bornological coarse space $\Gcan$ and the chain complexes
		computing Hochschild and cyclic homology of the $k$-algebra~$k[G]$.	
	\end{enumerate}
\end{theointra}

The construction of the functors $\XHH_k^\G$ and $\XHC_k^\G$ uses a cyclic homology theory for dg-categories that satisfies certain additive and localizing properties in the sense of Tabuada~\cite{tabuada}.  This   is \emph{Keller's cone construction}
$$
\fMix\colon \mathbf{dgcat}\to\Mix,
$$
for dg-categories \cite{keller},
defined as  a functor
from the category $\mathbf{dgcat}$ of small dg-categories to Kassel's category~$\Mix$ of mixed complexes \cite{kassel}. Hochschild and cyclic homologies for dg-categories  are then defined  in terms of mixed complexes, consistently with the classical definitions for $k$-algebras \cite{kassel}. We also consider the functor 
(with values in the category of small $k$-linear categories $\Cat_{k}$)
$$
V_k^G\colon G\BC\to \Cat_{k},
$$
that associates to every $G$-bornological coarse space $X$ a suitable $k$-linear category $V_k^G(X)$ of $G$-equivariant $X$-controlled (finite-dimensional) $k$-vector spaces \cite[Def.~8.3]{bunkeeq}; a $k$-linear category is a dg-category in a standard way.
We prove that the following functor 
\begin{equation*}
\begin{tikzcd}
\XC_k^\G\colon \G\BC \arrow{r}{V_{k}^{\G}} & \mathbf{Cat}_{k}\arrow{r}{\iota}&\dgcat_{k} 
\arrow{r}{\fMix} & \Mix\arrow{r}{\loc}&\Mix_\infty 
\end{tikzcd}
\end{equation*}
(see Definition \ref{defXHH})  with values in the cocomplete stable $\infty$-category of mixed complexes $\Mix_{\infty}$  is a coarse homology theory (Theorem \ref{echt}).
Coarse Hochschild  $\XHH_k^\G$ and coarse cyclic homology $\XHC_k^\G$ are then  defined by post-composition of the  Hochschild and cyclic homology functors for mixed complexes with the functor  $\XC_k^G$  (see Definition \ref{eqcHH}).\\

Let $\Sp$ be the $\infty$-category of spectra.
The main reason of defining coarse versions of Hochschild and cyclic homology  is to relate them to (the $\Sp$-valued) equivariant coarse algebraic $K$-homology $\cX K_{k}^G\colon G\BC\to\Sp$ \cite[Def.~8.8]{bunkeeq}, because classically algebraic $K$-theory comes equipped with trace maps (\emph{e.g.,} the Dennis trace map from algebraic $K$-theory of rings to Hochschild homology, or the refined version, the cyclotomic trace, from the algebraic $K$-theory spectrum to the topological cyclic homology spectrum) to cyclic homology theories, and these trace maps have been of fundamental importance in its understanding~\cite{MR1202133, mcdg}. 
Inspired by the classical case, we define  trace maps   to  equivariant coarse Hochschild and cyclic homology and from equivariant coarse Hochschild and cyclic homology to equivariant coarse ordinary homology $\XH^G\colon G\BC\to \Sp$ (see 
Proposition~\ref{compkth}, Theorem \ref{suzfagsazgsafg} and Proposition~\ref{kajsbkbvk}):

\begin{theointrb}
	\begin{enumerate}
		\item The classical Dennis trace map  induces a  natural transformation of equivariant coarse homology theories: $$\mathrm{K}\mathcal{X}_{k}^\G\to\XHH_{k}^\G;$$ 
		\item 	There is a natural transformation 
		\begin{equation*}
		\Phi_{\XHH_{k}^\G}\colon \XHH_{k}^\G\to\XH^\G_{}
		\end{equation*}
		of $G$-equivariant  coarse homology theories, which induces an  equivalence of spectra when evaluated at the one-point space $\{*\}$.
	\end{enumerate}
\end{theointrb}

By composition, we get a natural transformation, 
$$K\X_k^G\to\XHH_k^G\xrightarrow{} \XH^G$$
that factors through coarse Hochschild homology. The advantage of this transformation is that equivariant coarse ordinary homology $\XH^G$  is defined in terms of equivariant locally finite controlled maps $ X^{n+1}\to k$ (see Definition~\ref{nchaincontrolled}) and it might be suitable for computations of coarse $K$-theory classes. \\

We conclude with some applications to assembly maps.
One of the main applications of coarse homotopy theory is within the studying of assembly map conjectures. We then provide a comparison result between the forget-control maps  for equivariant coarse Hochschild  and cyclic  homology and the associated assembly maps (see Proposition~\ref{assembly}).

\subsection*{Structure of the paper}
In Section \ref{sub11} we review the basic definitions in coarse homotopy theory: bornological coarse spaces, coarse homology theories and categories of controlled objects. In Section \ref{jlahljhgölah}, we introduce the (cocomplete stable $\infty$-category) of mixed complexes and Keller's definition of cyclic homology. In Section~\ref{ahkbfkjabljavö} we define the functors $\XC_k^G$, $\XHH_k^G$ and $\XHC_k^G$ and we prove that they are equivariant coarse homology theories. In the last section Section \ref{hasglkfhagsakö}, we construct the natural transformations from coarse algebraic $K$-homology to  coarse ordinary homology factoring through coarse Hochschild homology.

\subsection*{Conventions}
We freely employ the language of  $\infty$-categories. More precisely, we model $\infty$-categories as quasi-categories \cite{cisin,htt,HA}.  When not otherwise specified, $G$ will denote a  group, $k$ a field, $\o$ the tensor product over $k$. Without further comments, we always consider an additive category as a dg-category in the canonical way.

\section{Equivariant coarse homotopy theory}\label{sub11}

The main purpose of this section is to recollect the  basic definitions in coarse homotopy theory and the notations  needed in  Section~\ref{ahkbfkjabljavö} and Section \ref{hasglkfhagsakö}. We describe the category $\G\mathbf{BornCoarse}$ of $\G$-equivariant bornological coarse spaces and the associated $\G$-equivariant coarse homology theories, we give the examples of coarse ordinary homology and coarse algebraic $K$-homology, together with  the properties of the category of controlled objects $V_k^G(X)$. We refer to \cite[Sec.~2]{bunke} and \cite[Sec.~2 \& Sec.~3]{bunkeeq}  for a comprehensive introduction to (equivariant) coarse homotopy theory.

\subsection{Equivariant bornological coarse spaces}

A \emph{bornology} on a set $X$ is a subset $\B\subseteq \P(X)$ of the power set of $X$ that is closed under taking subsets and finite unions, and such that $X=\cup_{B\in\B} B$. Its elements are called \emph{bounded sets}.

A \emph{coarse structure} on a set $X$ is a subset $\CC\subseteq \P(X\x X)$ which contains the diagonal 
$\Delta_X\coloneqq \{(x,x)\in X\x X\mid x\in X\}$ and is closed under taking subsets, finite unions, inverses, and compositions. 
The elements of $\CC$ are called \emph{entourages}.
If $\mathcal{U}$ is an entourage of a coarse space $X$ and $B$ is any subset of $X$,  the \emph{$\mathcal{U}$-thickening
	of $B$} is the subset of $X$:
\begin{equation}\label{thickening}
\mathcal{U}[B]\coloneqq\{x\in X \mid \exists b\in B, \; (x,b)\in \mathcal{U}\}\subseteq X
\end{equation}

A bornology $\B$ and a coarse structure $\CC$ on a set $X$ are  \emph{compatible} if for every $\U\in \CC$ and every $B\in\B$ the controlled thickening $\U[B]$ belongs to the family $\B$. 

\begin{df}\cite[Def.\,2.5]{bunke}\label{BC}
	A \emph{bornological coarse space} is a triple $(X,\CC,\B)$ given by a set $X$, a bornology $\B$ and a coarse structure $\CC$ on $X$, such that 
	$\B$ and  $\CC$ are compatible. 
\end{df}

Morphisms of bornological coarse spaces are maps such that pre-images of bounded sets are bounded sets and images of entourages are entourages.  
A $G$-bornological coarse space \cite[Def.~2.1]{bunkeeq} is a bornological coarse space $(X,\CC,\B)$ equipped with a $G$-action by automorphisms such that the set of $G$-invariant entourages $\CC^G$ is cofinal in $\CC$.
We denote by $G\BC$ the category of $G$-bornological coarse spaces and $\G$-equivariant, proper  controlled maps. When clear from the context, we shortly write  $X$ for denoting a $G$-bornological coarse space $(X,\CC,\B)$.

\begin{ex}\label{equivex}
	
	\begin{enumerate}[(i)] 
		\item
		\label{jsdlkhbskdl} Let $\G$ be a group, $\B_{\min}$ be the minimal bornology on its underlying set and let $\CC_{\mathrm{can}}\coloneqq \langle \{\G(B\x B)\mid B\in\B_{\min}\}\rangle$ be the coarse structure on  $G$ generated by the $\G$-orbits. The space $\G_{\mathrm{can,min}}\coloneqq(\G,\CC_{\mathrm{can}},\B_{\min} )$ is a $\G$-bornological coarse space. 
		\item Let $X$ be a $\G$-bornological coarse space and let $Z$ be a $\G$-invariant subset of $X$. Then, the triple $Z_X\coloneqq (Z,\CC_Z,\B_Z)$ is a $\G$-bornological coarse space, where 
		$\CC_Z\coloneqq \{(Z\x Z)\cap U\mid U\in \CC\}$ and $\B_Z\coloneqq\{Z\cap B\mid B\in \B \}$. 
		\item Let 
		$\U$ be a $\G$-invariant entourage of $X$.  If 
		$\CC_\U$ denotes the coarse structure on $X$ generated by $\U$, then 
		$X_\U\coloneqq (X,\CC_\U,\B)$ is a $\G$-bornological coarse space. Observe that there is a canonical morphism $X_U\to X$.
	\end{enumerate}
\end{ex}

\subsection{Equivariant coarse homology theories}\label{sub12}

Let $f_0,f_1\colon X\to X'$ be morphisms between bornological coarse spaces.
We say that $f_0$ and $f_1$ are \emph{close} to each other  if the image of the diagonal  $(f_0,f_1)(\Delta_X)$ is an entourage of $X'$.
A morphism $f\colon X\to X'$ is an \emph{equivalence}  if there exists an inverse $g\colon X'\to X$ such that the compositions $g\circ f$ and $f\circ g$ are close to the respective identity maps. 
Two morphisms between $\G$-bornological coarse spaces are close to each other if they are close as morphisms between the underlying bornological coarse spaces.

\begin{df}\cite[Def.\,3.8]{bunkeeq}\label{flasque}
	A  $\G$-bornological coarse space $(X,\CC,\B)$ is called \emph{flasque} if it admits a morphism $f\colon X\to X$ such that:
	\begin{enumerate}[(i)]
		\item $f$ is close to the identity map;
		\item for every entourage $\U$, the subset $\bigcup_{k\in \N} (f^k\x f^k)(\U)$ is an entourage of $X$; 
		\item for every bounded set $B$ in $X$ there exists  $k$ such that $f^k(X)\cap \G B=\emptyset$.
	\end{enumerate}
\end{df}

\begin{df}\cite[Def.\,3.5 \& 3.7]{bunkeeq}\label{complementary}
	Let $(X,\CC,\B)$ be a $G$-bornological coarse space.
	\begin{enumerate}
		\item A \emph{big family} $\Y=(Y_i)_{i\in I}$ on $X$ is a filtered family of subsets of $X$ satisfying the following:
		\[
		\forall\, i\in I, \quad \forall\, \U\in\CC, \, \exists j\in I \text{ such that } \U[Y_i]\subseteq Y_j
		\] 
		An \emph{equivariant big family}  is a big family consisting of $\G$-invariant subsets. 
		\item A pair $(Z,\Y)$ consisting of a subset $Z$ of $X$ and of a big family $\Y$ on $X$ is called a \emph{complementary pair} if there exists an index $i\in I$ for which $Z\cup Y_i=X$. It is 
		an \emph{equivariant complementary pair} if  $Z$ is a $\G$-invariant subset and $\Y$ is an equivariant big family.
	\end{enumerate}
\end{df}

Let $Z$ be a subset of $X$.  If $\Y$ is a big family on $X$, then the intersection $Z\cap \Y\coloneqq (Z\cap Y_i)_{i\in I}$ is a big family on $Z$. 
If $E\colon \G\BC\to \mathbf{C}$ is a functor with values in a cocomplete $\infty$-category  $\mathbf{C}$,  
we define the value of $E$ at the family $\Y$ as the filtered colimit $E(\Y)\coloneqq\mathrm{colim}_{i\in I} E(Y_i)$. There is an induced map from $E(\Y)$ to $E(X)$.  Let $X_U$ be the $G$-bornological coarse space constructed in Example \ref{equivex}.

\begin{df}\cite[Def.\,3.10]{bunkeeq}\label{hcht}
	Let $G$ be a group and let $G\BC$ be the category of $G$-bornological coarse spaces.
	Let $\mathbf{C}$ be a cocomplete stable $\infty$-category. 
	A $\G$-\emph{equivariant $\mathbf{C}$-valued coarse homology theory} is a functor 
	\[
	E\colon \G\BC \longrightarrow \mathbf{C}
	\]
	with the following properties:
	\begin{enumerate}[i.]
		\item \textbf{Coarse invariance:}  $E$ sends equivalences $X\to X'$ of $\G$-bornological coarse spaces  to equivalences  $E(X)\to E(X')$ of $\mathbf{C}$;
		\item \textbf{Flasqueness:} if $X$ is a flasque $\G$-bornological coarse space, then $E(X)\simeq 0$;
		\item \textbf{Coarse excision:} $E(\emptyset)\simeq 0$, and for every equivariant complementary pair $(Z,\Y)$ on  $X$, the diagram
		\[
		\begin{tikzcd}
		E(Z\cap \Y) \arrow{d} \arrow{r} & E(Z) \arrow{d} \\
		E(\Y) \arrow{r} & E(X)  
		\end{tikzcd}
		\]
		is a push-out square;
		\item \textbf{u-continuity:} for every $\G$-bornological coarse space $(X,\CC,\B)$, the canonical morphisms $X_\U\to X$ induce an equivalence $E(X)\simeq \colim_{\U\in \CC^\G}E(X_\U )$.
	\end{enumerate}
\end{df}

Examples of (equivariant) coarse homology theories are coarse ordinary homology \eqref{kjdhlkvsd} and coarse topological  K-theory \cite{bunke}, coarse algebraic K-theory (Definition \ref{sgualau}) and coarse topological Hochschild homology \cite{bunkeeq,bunkecisinski}, coarse Hochschild and cyclic homology. 

\subsection{Coarse ordinary homology}

Let $X$ be a $G$-bornological coarse space, $n\in \N$ a natural number, $B $ a bounded set of $X$, and $x=(x_0,\dots,x_n)$  a point of $X^{n+1}$. We say that $x$ \emph{meets} $B$ if there exists $i\in \{ 0,\dots,n \} $ such that $x_{i}$ belongs to $ B$. If $\U$ is an entourage of $X$, we say that $x$ is \emph{$\U$-controlled} if, for each $i$ and $j$, the pair $(x_i,x_j)$ belongs to  $\U$. 

An  \emph{$n$-chain} $c$ on $X$ is a function 
$
c\colon X^{n+1}\to \Z 
$; its \emph{support} $\supp(c)$ is defined as the set of points for which the function $c$ is non-zero: 
\begin{equation}\label{suppordhom}
\supp(c)=\{x\in X^{n+1}\mid c(x)\neq 0\}.
\end{equation}
We say that an  $n$-chain $c$ is \emph{$\U$-controlled} if every point $x$ of $\supp(c)$ is $\U$-controlled. The chain $c$  is \emph{locally finite} if, for every bounded set $B$, the set of points in $\supp(c)$ which meet $B$ is finite. An $n$-chain $c\colon X^{n+1}\to \Z$ is \emph{controlled} if it is locally finite and $\U$-controlled for some entourage $\U$ of $X$.

\begin{df}\label{nchaincontrolled}
	Let $X$ be a bornological coarse space.
	Then, for $n\in \N$,  $\X C_n(X)$ denotes the free abelian group generated by the locally finite controlled $n$-chains on $X$.
\end{df}  

We will also represent $n$-chains as formal sums 
$\sum_{x\in X^{n+1}} c(x) x$  that are locally finite and $\U$-controlled.
The boundary map $\partial\colon \X C_n(X)\to \X C_{n-1}(X)$ is defined as the alternating sum $\partial\coloneqq\sum_i(-1)^i\partial_i$ of the face maps $\partial_i(x_0,\dots,x_n)\coloneqq (x_0,\dots,\hat{x_i},\dots,x_n)$.
The graded abelian group $\X C_*(X)$, endowed with the boundary operator $\partial$ extended linearly to $\X C_{*}(X)$, is a chain complex \cite[Sec.~6.3]{bunke}.
When $X$ is a $\G$-bornological coarse space, we let $\X C_n^\G(X)$ be the subgroup of $\X C_n(X)$ given by the locally finite controlled $n$-chains that are also $\G$-invariant. The boundary operator restricts to  $\X C_*^\G(X)$, and $(\X C_*^\G(X),\partial)$ is a subcomplex of $(\X C_*(X),\partial)$.

If $f\colon X\to Y$ is a morphism of $G$-bornological coarse spaces, then  we consider the map on the products $X^{n+1}\to Y^{n+1}$ sending $(x_{0},\dots,x_{n})$ to $(f(x_{0}),\dots,f(x_{n})) $. It extends linearly to a map
$
\X C^G(f)\colon \X C^G_{n}(X)\to \X C^G_{n}(Y)
$
that  involves  sums over the pre-images by $f$.
This describes  a functor 
$\label{guafglagf}\X C^G\colon G\BC\to\mathbf{Ch}$ 
with values in the  category $\bCh$ of chain complexes  over the integers. 
The $\infty$-category $\bCh_\infty$ of chain complexes is defined as the localization (in the realm of $\infty$-categories  \cite[Sec.~1.3.4]{HA}) of the nerve of the category $\bCh$ at the class $W$ of quasi-isomorphisms of chain complexes
$
\bCh_\infty\coloneqq \Nerve_{}(\bCh)[W^{-1}].
$
By post-composing the functor $\X C^\G$ with the functor 
\begin{equation}\label{EMcorr}
\mathcal{EM}\colon \bCh\xrightarrow{\loc} \mathbf{Ch}_{\infty}\xrightarrow{\simeq} H\Z\mhyphen\Mod\to \Sp
\end{equation}
(the Eilenberg-MacLane correspondence between chain complexes and spectra \cite[Thm.~1.1]{MR2306038} or  \cite[Sec.~6.3]{bunke}),  we get a functor to the $\infty$-category of spectra
\begin{equation}\label{kjdhlkvsd}
\XH^\G\coloneqq \mathcal{EM}\circ \X C^\G\colon \G\BC\to\Sp\end{equation} 
called \emph{equivariant coarse ordinary homology}:  

\begin{thm}\cite[Thm.\,7.3]{bunkeeq}
	The  functor  $\XH^\G$ is a $G$-equivariant $\Sp$-valued coarse homology theory.
\end{thm}

\begin{ex}\label{coarseordinarypoint}
	If $X$ is a point, its coarse homology groups are $0$ in positive and negative degree and  the base ring $k$ in degree $0$. 
\end{ex}

\subsection{The category $V^\G_\bA(X)$ of controlled objects}\label{cobj}

The goal of this subsection is to recall the definition of the additive category $V^\G_\bA(X)$ of $G$-equivariant  $X$-controlled $\bA$-objects \cite[Def.~8.3]{bunkeeq} and of the  functor 
$
V_\bA^G\colon G\BC\to \Add
$
sending a $G$-bornological coarse space to the category $V^\G_\bA(X)$. This functor is an essential ingredient in the construction of coarse homology theories like coarse algebraic $K$-homology and coarse Hochschild and cyclic homology.

Let $G$ be a group and let $X$ be a $G$-bornological coarse space.
\begin{rem}\label{jagshhgl}
	The bornology $\B(X)$ on $X$ defines a poset with the partial order induced by subset inclusion; hence, $\B(X)$ can be seen as a category. 
\end{rem}

Let $\bA$ be an additive category with strict $G$-action.
For every element $g$ in $\G$ and every functor $F\colon \B(X)\to \bA$, let $g F\colon\B(X)\to\bA$ denote the functor  sending a bounded set $B$ in $\B(X)$ to the $\bA$-object  $g(F(g^{-1}(B)))$ (and defined on morphisms $B\subseteq B'$ as the induced morphism of $\bA$ $(gF)(B\subseteq B')\colon gF(g^{-1}(B))\to gF(g^{-1}(B'))$). 

If  $\eta\colon F\to F'$ is a natural transformation between two functors $F,F'\colon \B(X)\to \bA$, we denote by $g\eta\colon g F\to g F'$ the induced natural transformation between $gF$ and $gF'$.

\begin{df}\label{controbj}\cite[Def.\,8.3]{bunkeeq}
	A \emph{$\G$-equivariant $X$-controlled $\bA$-object} is a pair $(M,\rho)$  consisting of a functor $M\colon \B(X)\to \bA$ and a family $\rho=(\rho(g))_{g\in \G}$ of natural isomorphisms $\rho(g)\colon M\to g M$, satisfying the following conditions:
	\begin{enumerate}
		\item $M(\emptyset)\cong 0$;
		\item for all $B,B'$ in $\B(X)$, the commutative diagram 
		\[
		\begin{tikzcd}
		M(B\cap B') \arrow{d} \arrow{r} & M(B) \arrow{d} \\
		M(B') \arrow{r} & M(B\cup B')
		\end{tikzcd}
		\]
		is a push-out;
		\item for all $B$ in $\B(X)$ there exists a finite subset $F$ of $B$ such that the inclusion induces an isomorphism $M(F)\xrightarrow{\cong} M(B)$;
		\item for all elements $g, g'$ in $G$ we have the relation  $\rho(g g')=g\rho(g')\circ \rho(g)$, where $g\rho(g')$ is the natural transformation from $g M$ to $g g' M$ induced by $\rho(g')$.
	\end{enumerate}
\end{df}

\begin{notation}
	If $(M,\rho)$ is an $X$-controlled $\bA$-object and $x$ is an element of $X$, we will often write $M(x)$ instead of $M(\{x\})$ for the value of the functor $M$ at the bounded set $\{x\}$ of $X$.
\end{notation}

Let $X  $ be a  $G$-bornological coarse space    and let $(M,\rho)$ be an equivariant $X$-controlled $\bA$-object. Let $B$ be a bounded set of $X$ and let $x$ be a point in $B$. The inclusion $\{x\}\to B$ induces a morphism
$M(\{x\})\to M(B)$ of $\bA$. The 
conditions of Definition \ref{controbj}
imply that
$M(\{x\})=0$ for all but finitely many points of $B$ and that the canonical morphism (induced by the universal property of the coproduct in $\bA$)
\begin{equation}\label{g5giojgoigg5g45gfg}
\bigoplus_{x\in B} M(\{x\}) \xrightarrow{\cong} M(B)
\end{equation}   
is an isomorphism.
The $\U$-thickening $\U[B]$ \eqref{thickening} of a bounded subset $B$ of $X$ is bounded and $U$-thickenings preserve the inclusions of bounded sets; we get a functor
$
\U[-]\colon\B(X)\to\B(X)
$.

\begin{df}\cite[Def.\,8.6]{bunkeeq}\label{Uequivmorh}
	Let  $(M,\rho)$ and $(M',\rho')$ be   $G$-equivariant $X$-controlled $\bA$-objects and let $\U\in\CC^\G(X)$ be a $G$-invariant entourage of $X$. 
	A \emph{$\G$-equivariant $\U$-controlled morphism} $\phi\colon (M,\rho) \to (M',\rho')$ is a natural transformation 
	\[
	\phi\colon M(-)\to M'\circ \U[-]
	\]
	such that $\rho'(g)\circ \phi=(g \phi)\circ \rho(g) $ for all $g$ in $\G$. 
\end{df}

The set of $G$-equivariant $\U$-controlled morphisms $\phi\colon (M,\rho)\to (M',\rho')$ is denoted by $\Mor_\U((M,\rho),(M',\rho'))$.
For every bounded set $B$ of $X$, the inclusion $\U\subseteq \U'$ induces  an inclusion $\U[B]\subseteq \U'[B]$; this yields  
a natural transformation of functors 
$M'\circ \U[-] \to M'\circ \U'[-]$,
hence a map $$\Mor_\U((M,\rho),(M',\rho'))\to \Mor_{\U^{\prime}}((M ,\rho),(M^{\prime},\rho')) $$ by post-composition.

By using these structure maps we define the abelian group of \emph{$\G$-equivariant controlled morphisms} from $(M,\rho)$ to $(M',\rho')$ as the colimit
\[
\Hom_{V_{\bA}^G(X)}((M,\rho),(M',\rho'))\coloneqq \colim_{\U\in \CC^\G} \mbox{Mor}_\U((M,\rho),(M',\rho')).\
\]

\begin{df}\cite{bunkeeq}\label{catcontrobj}
	Let $X$ be a $G$-bornological coarse space and let $\bA$ be an additive category with strict $G$-action.
	The category $V^\G_\bA(X)$ is the category of $\G$-equivariant $X$-controlled $\bA$-objects and $G$-equivariant controlled morphisms.
\end{df}

Let $k$ be a field. When $\bA$ is the category of finite-dimensional $k$-vector spaces,  then we denote by $V_k^G(X)$ the associated category of $G$-equivariant $X$-controlled (finite-dimensional) $k$-modules.

\begin{lemma}\cite[Lemma 8.7]{bunkeeq}\label{aouhoasöghogh}
	The category of equivariant $X$-controlled $\bA$-objects $V^\G_\bA(X)$ is additive. 
\end{lemma}

Let $f\colon (X,\CC,\B)\to (X',\CC',\B')$ be a morphism of $\G$-bornological coarse spaces. If $(M,\rho)$ is a $\G$-equivariant $X$-controlled $\bA$-object, we consider the functor 
$
f_* M\colon \B'\to\bA
$
defined by $f_* M(B')\coloneqq M(f^{-1}(B'))$ for every bounded set 
$B'$ in $\B'$ and defined on morphisms in the canonical way. For every $g$ in $G$,  the family of transformations  $f_{*}\rho=((f_{*}\rho)(g))_{g\in G}$ is given by  the natural isomorphisms $(f_{*}\rho)(g)\colon f_{*}M\to g(f_{*}M)$ with
$((f_*\rho)(g))(B')\coloneqq \rho(g)(f^{-1}(B')).$
The pair $f_*(M,\rho)\coloneqq (f_*M,f_*\rho)$ defined in this way is a $\G$-equivariant $X'$-controlled $\bA$-object \cite[Sec.~8.2]{bunkeeq}.
Assume also that $\U$ is an invariant entourage of $X$ and that $\phi \colon (M,\rho) \to (M',\rho')$ is an equivariant $\U$-controlled morphism. Then, the set $V:= (f \times f)(\U)$ is a $G$-invariant entourage of $X'$ and 
the morphism:
\begin{equation}\label{hgsalfglagfl}
f_*\phi \coloneqq  \left(  f_*M(B^{\prime})\xrightarrow{\phi_{f^{-1}(B^{\prime})}} M(\U[f^{-1}(B^{\prime})])\to  f_*M(V[B^{\prime}]) \right)_{B^{\prime} \in \B'}\end{equation}
is $V$-controlled.
We have just described  a functor
$
f_* \coloneqq V^G_{\bA}(f)\colon V^\G_\bA(X) \to V^\G_\bA(X').
$

We denote by  \begin{equation}\label{V}
V_\bA^\G\colon\G\BC\to \mathbf{Add}.
\end{equation} the functor from the category of $\G$-bornological coarse spaces to the category of small additive categories obtained in this way.

\begin{rem}\label{ljsahaòlsghgjòal}
	If $\bA$ is a $k$-linear category,  then the functor $V_\bA^\G\colon\G\BC\to \mathbf{Add}$ refines to a functor $V_\bA^\G\colon\G\BC\to \mathbf{Cat}_{k}$ from the category of $G$-bornological coarse spaces to the category of small $k$-linear categories.
\end{rem}

The following properties of the functor $V^G_\bA$ are shown  in \cite{bunkeeq}:

\begin{rem}\label{colimits}
	Let $(X,\CC,\B)$ be a $\G$-bornological coarse space, $\U\in\CC^\G$  a $\G$-invariant entourage of $X$ and  $X_\U\coloneqq (X,\CC_\U,\B)$ the $\G$-bornological coarse space obtained by restriction of the structures. 
	Then, the category $V_\bA^\G(X)$ is the filtered colimit
	\[
	V_\bA^\G(X)\cong \colim_{\U\in\CC^\G}V_\bA^\G(X_\U)
	\]
	indexed on the poset of $\G$-invariant entourages of $X$.
\end{rem}

\begin{lemma}\cite[Lemma 8.11]{bunkeeq}\label{coarseinvariance}
	Let $f,g\colon X\to X'$ be two morphisms of $\G$-bornological coarse spaces. 
	If $f$ and $g$ are close to each other, then they induce naturally isomorphic  functors $f_*\cong g_*\colon V^G_\bA(X)\to V_\bA^G(X')$.
\end{lemma}

Let $\A$ be an additive category and denote by $\oplus$ its biproduct. 
Recall that  $\A$ is called \emph{flasque} if it admits an endofunctor 
$S\colon \A\to \A$ and a natural isomorphism $\mathrm{id}_\A\oplus S\cong S$.

\begin{lemma}\cite[Lemma 8.13]{bunkeeq}\label{flasquelemma}
	If $X$ is a flasque $\G$-bornological coarse space, then the category $V_\bA^\G(X)$ of $G$-equivariant $X$-controlled $\bA$-objects  is  a flasque category.
\end{lemma}

We conclude with the definition of coarse algebraic $K$-homology:

\begin{df}\cite[Def.\,8.8]{bunkeeq}\label{sgualau}
	Let $G$ be a group and let $\bA$ be an additive category with strict $G$-action.
	The \emph{$G$-equivariant coarse algebraic $K$-homology} associated to $\bA$ is the $K$-theory of the additive category of $\bA$-controlled objects:
	\[
	K\bA\cX^G\coloneqq K\circ V_{\bA}^G\colon G\BC\to\Sp.
	\] 
\end{df}

When  $\bA$ is the category of finite-dimensional $k$-vector spaces, we denote by $K\cX_{k}^G$ the associated $K$-theory functor.
The properties of the functor $V_{\bA}^G$ reviewed above are used in order to prove the following:

\begin{thm}\cite[Thm.\,8.9]{bunkeeq}
	Let $G$ be a group and let $\bA$ be an additive category with strict $G$-action.
	Then, the functor $K\bA\cX^G$ is a $G$-equivariant $\Sp$-valued coarse homology theory.
\end{thm}

\section{Keller's cyclic homology for dg-categories}\label{jlahljhgölah}

In this section  we recall Keller's construction of cyclic homology for dg-categories \cite{keller}. We start by recalling some properties of differential graded categories and mixed complexes, we introduce  Keller's construction and then review Keller's Localization Theorem \cite[Thm.~1.5]{keller}.
Keller defines the cyclic homology of a dg-category as the cyclic homology of a suitable mixed complex associated to it.
We point here that,  in the next Section \ref{ahkbfkjabljavö} and in particular in Definition \ref{defXHH}, we will need Keller's cyclic homology in the less general  context of additive ($k$-linear) categories. However, for consistency with his language and for sake of completeness, we will state Keller's definition and results in the broader context of dg-categories.

\subsection{Dg-categories}
In the following, we use the same conventions on differential graded categories and their properties as found in \cite{kellerdg}; we refer to the same survey for a general overview on the subject.
We recall that a dg-category over $k$ is a category enriched on (the category of) chain complexes of $k$-modules and that every additive, or $k$-linear, category, is a dg-category in a canonical way. We denote by $\mathbf{dgcat}_k$ the category of small dg-categories (over $k$) and dg-functors.

\begin{rem}\label{jganlngjòl}
	The category of dg-modules (over a dg-algebra or a dg-category) admits two Quillen model structures  where the weak equivalences are the objectwise quasi-isomorphisms of dg-modules; these are the injective and the projective model structure  induced from the injective and projective model structure  on chain complexes, respectively.
	We remark that the category of dg-modules over a dg-algebra, equipped with the projective model structure (hence the fibrations are the objectwise epimorphisms), is a combinatorial model category; see, for example, \cite[Rem.~2.14]{cohn}.
\end{rem}
If $\A$ is a dg-category, we can define an associated derived category:

\begin{df}\label{dercat}\cite[Sec.\,3.2]{kellerdg}
	The \emph{derived category}~$\D(\A)$ of a dg-category $\A$ is the localization of the category of dg-modules over $\A$ at the class of quasi-isomorphisms.
\end{df}

The objects of $\D(\A)$ are the dg-modules over $\A$ and the morphisms are obtained from morphisms of dg-modules by inverting the quasi-isomorphisms.  It is a triangulated category with shift functor induced by the $1$-translation and triangles coming from short exact sequences of complexes.

Let $\A$ and $\B$ be two small dg-categories. A dg-functor $F\colon \A\to\B$ is called a \emph{Morita equivalence} if it induces an equivalence of derived categories. For a precise definition of Morita equivalences we refer to \cite[Sec.~3.8]{kellerdg}, or \cite[Def.~2.29]{cohn}. 
\begin{thm}\cite[Thm.\,5.1]{tabuada2}\label{thmmorita}
	The category $\bdgcat_k$ of small dg-categories over $k$ admits the structure of a combinatorial model category whose weak equivalences are the Morita equivalences.
\end{thm}

For a description of fibrations and cofibrations we refer to \cite[Thm.~5.1]{tabuada2}, or \cite[Thm.~4.1]{kellerdg}. 
We conclude with the definition of short exact sequences of dg-categories:

\begin{df}\cite[§4.6]{kellerdg}\label{kjasfkhasgfksal}
	A \emph{short exact sequence} of dg-categories is a sequence of morphisms $\A\to\B\to\CC$ inducing an exact sequence of triangulated categories
	\[
	\D^b (\A)\to \D^b (\B)\to \D^b (\CC)
	\]
	in the sense of Verdier.
\end{df}

\subsection{The $\infty$-category of mixed complexes}
\label{secmixcompl}

In this subsection we describe the (cocomplete stable $\infty$-)category of unbounded mixed complexes.
We follow Kassel's approach \cite{kassel}.

\begin{df}\cite[§1]{kassel}\label{mixcompl}
	A \emph{mixed complex} $(C,b,B)$ is a triple consisting  of a  $\Z$-graded $k$-module $C=(C_p)_{p\in \Z}$ together with differentials $b$ and $B$  $$b=(b_{p}\colon C_p\to C_{p-1})_{p\in\Z} \quad  \text{ and } \quad  B=(B_{p}\colon C_p\to C_{p+1})_{p\in\Z}$$  of degree $-1$ and $1$, respectively, satisfying the following identities:
	\[
	b^2=0,\quad B^2=0,\quad bB+Bb=0.
	\]
	Morphisms   of mixed complexes  are given by maps commuting with both the differentials $b$ and $B$. The category of mixed complexes and morphisms of mixed complexes is denoted by $\Mix$.
\end{df}

When the differentials are clear from the context, we refer to a mixed complex $(C,b,B)$ by its underlying $k$-module $C$.

Let $\Lambda$ be the dg-algebra  over the field $k$
\begin{equation}
\Lambda\coloneqq \dots\to 0\to k\epsilon\xrightarrow{0} k\to 0\to\cdots
\end{equation}
generated by an indeterminate $\epsilon$ of degree $1$, with $\epsilon^2=0$ and differential (of degree $-1$) $d(\epsilon)=0$.
Mixed complexes are nothing but dg-modules over the  dg-algebra $\Lambda$:

\begin{rem}\cite{kassel}\label{agsfafyl}
	The  category $\Mix$ of mixed complexes is equivalent (in fact, isomorphic) to the category of left $\Lambda$-dg-modules, which we denote by $\Lambda\text{-}\Mod$. 
	We denote by  $L\colon\Mix\to\LMod$ the  functor sending a mixed complex to the associated $\Lambda$-dg-module and by $R\colon \LMod\to\Mix$  its inverse functor.
\end{rem}

The category of $k$-dg-modules  admits a combinatorial model structure (the projective model structure, see Remark \ref{jganlngjòl}), whose weak equivalences are the objectwise quasi-isomorphisms of dg-modules. In the language of mixed complexes this translates as follows:

\begin{df}\label{quasiisomixcoplx}
	A morphism $(C,b,B)\to (C',b',B')$ of mixed complexes is called a \emph{quasi-isomorphism} if the underlying $b$-complexes are quasi-isomorphic via the induced chain map $(C,b)\to(C',b')$. 
\end{df}

\begin{rem}\label{lkghòadshò}
	Quasi-isomorphisms of mixed complexes correspond to quasi-isomorphisms of $\Lambda$-dg-modules, \emph{i.e.,} the functors  $L$ and $R$ of Remark \ref{agsfafyl} preserve  quasi-isomorphisms. 
\end{rem}

If $\bC$ is an ordinary category and $W$ denotes a collection of morphisms of $\bC$, then $\Nerve_{}(\bC)[W^{-1}]$ is the $\infty$-category obtained by the nerve $\Nerve(\bC)$ of $\bC$ by inverting the set of morphisms $W$ (see \cite[Def.~7.1.2 \& Prop.~7.1.3]{cisin}, \cite[Def.~1.3.4.1]{HA}).

\begin{df}\label{lmodinfty}
	The $\infty$-category \begin{equation*}
	\Mix_\infty\coloneqq \Nerve_{}(\Mix)[W_{\mathrm{mix}}^{-1}]
	\end{equation*} of mixed complexes is  defined as the localization  of the (nerve of the) category $\Mix$ at the class $W_{\mathrm{mix}}$ of quasi-isomorphisms of mixed complexes.
\end{df}

Analogously, the $\infty$-category $\LMod_\infty$ is  defined as the localization of the category $\LMod$ of $\Lambda$-dg-modules  at the class $W$ of quasi-isomorphisms of $\Lambda$-dg-modules:
\begin{equation}
\LMod_\infty\coloneqq \Nerve(\LMod)[W^{-1}].
\end{equation}

\begin{prop}\label{mixstconpl}
	The $\infty$-category $\Mix_\infty$ is a  cocomplete stable $\infty$-category. 
\end{prop}

\begin{proof}
	The category $\LMod$ is a (pre-triangulated) dg-category. By applying the dg-nerve functor $\Nerve_{\text{dg}}$ \cite[Constr.~1.3.1.6]{HA} 
	we obtain an $\infty$-category $\Nerve_{\text{dg}}(\LMod)$ \cite[Prop.~1.3.1.10]{HA}.
	The dg-nerve functor sends pre-triangulated dg-categories to stable $\infty$-categories \cite[Thm.~4.3.1]{faonte}, \cite[Prop.~1.3.1.10]{HA}. The $\infty$-category $\Nerve_{\text{dg}}(\LMod)$  is  stable  and its homotopy category can be identified (as a triangulated category) with the derived category $\D(\Lambda)$  associated to the dg-algebra $\Lambda$.  
	
	The category $\LMod$ is equipped with a combinatorial simplicial model structure by Remark \ref{jganlngjòl}. By \cite[Prop.~1.3.1.17]{HA} and by the fact that the simplicial nerve of the simplicial category associated to $\LMod$ is equivalent to the localization $\Nerve(\LMod)[W^{-1}]$ (by \cite[Rem.~1.3.4.16 \& Thm.~1.3.4.20]{HA} where we also use that the model category $\LMod$ is combinatorial, hence admits functorial factorizations), the two constructions $\Nerve(\LMod)[W^{-1}]$ and $\Nerve_{\text{dg}}(\LMod)$ present equivalent $\infty$-categories. Hence, the $\infty$-category $\LMod_{\infty}$ is a stable $\infty$-category.
	The $\infty$-category $\LMod_{\infty}$ is also cocomplete by \cite[Prop.~1.3.4.22]{HA} because the model category $\LMod$ is combinatorial.
	
	The categories $\Mix$ and $\LMod$  are isomorphic by  Remark \ref{agsfafyl}  and 
	the functor    $L\colon \Mix\to \LMod$ and its inverse $R\colon \LMod\to\Mix$ preserve quasi-isomorphisms by Remark \ref{lkghòadshò}. This yields an equivalence of $\infty$-categories 
	$$\Nerve_{}(\Mix)[W_{\mathrm{mix}}^{-1}]\to \Nerve_{}(\LMod)[W^{-1}]$$
	that proves the statement.
\end{proof}

\begin{rem}
	The homotopy category of the stable $\infty$-category  $\Mix_\infty$ is canonically equivalent to the derived category $\D(\Lambda)$ of the dg-algebra $\Lambda$.
\end{rem}

We conclude the subsection with the definition of Hochschild and cyclic homology of mixed complexes.
A mixed complex $(C,b,B)$ functorially determines a double chain complex  $\B C$ \cite[§2.5.10]{loday} by means of the differentials $b$ and $B$:
\begin{equation}\label{dcc}
\B C\coloneqq \left( \dots\xleftarrow{0} (C,b) \xleftarrow{B} (C[-1],b_{C[-1]}) \xleftarrow{B}\dots \xleftarrow{B} (C[-n],b_{C[-n]}) \xleftarrow{B}\dots \right);
\end{equation}
here, the chain complex $(C,b)$ is placed in bi-degree $(0,*)$, \emph{i.e.,} $\B C_{(0,*)}=(C_{*},b)$, and the chain complex  $(C[-n],b_{C[-n]})$,  placed in bi-degree $(n,*)$, is the  chain complex $(C,b)$ shifted by $-n$, hence $\B C_{(p,q)}=C_{q-p}$ for $p\geq 0$ and $\B C_{(p,q)}=0$ for $p<0$. The total chain complex $\Tot (\B C)$, functorially associated to the double chain complex $\B C$, is the chain complex defined in degree $n$ by $\Tot_n (\B C)=\bigoplus_{i\geq 0} C_{n-2i} $ with differential $d$ acting as follows: $d(c_n,c_{n-2},\dots)\coloneqq (bc_n + Bc_{n-2},\dots).$

Let $\bCh$ be the category  of chain complexes over $k$.
Consider the forgetful functor 
\begin{equation}\label{forgetf}
\mbox{forget}\colon\Mix\to \bCh
\end{equation}
sending a mixed complex $(M,b,B)$ to its underlying chain complex $(M,b)$, and the functor
\begin{equation}\label{totfunct}
\Tot(\B -)\colon \Mix \to \bCh
\end{equation}
just described above.

\begin{df}\cite[Sec.\,1]{kassel}\label{HHC}
	Let $(C,b,B)$ be a mixed complex. The \emph{Hochschild homology} $\HH_*(C)$ of  $(C,b,B)$ is  the homology of the underlying chain complex $(C,b)$. Its \emph{cyclic homology} $\HC_*(C)$  is the homology of the associated chain complex $\Tot(\B C)$. 
\end{df}

We remark that this definition agrees with the classical definition of Hochschild and cyclic homology of algebras \cite{kassel,loday}.

\subsection{Keller's cyclic homology }\label{secKeller}

Let $k$ be a commutative ring with identity and let $A$ be a $k$-algebra. Then, one can associate to $A$ a cyclic module $Z_*(A)$ \cite{goodwillie} (\emph{i.e.,} a cyclic object in the category of $k$-modules) defined in degree $n$ as the $(n+1)$-th tensor product of $A$ over $k$. In the same way, one can construct a cyclic module out of an additive category $\A$ \cite[Def.~2.1.1]{mccarthy}. We present these constructions in the more general setting of dg-categories.

\begin{df}\cite{keller}\label{acn}
	Let $\CC$ be a small dg-category over  $k$.
	The \emph{additive cyclic nerve} of  $\mathcal{C}$ is the cyclic $k$-module defined  by:
	\[
	\CN_n(\mathcal{C}) \coloneqq \bigoplus \Hom_\CC(C_1,C_0)\otimes\Hom_\CC(C_2,C_1)\otimes\dots\otimes \Hom_\CC(C_0,C_n)
	\] 
	where the sum runs over all the objects $(C_0,C_1,\dots,C_n)$ in $\mathcal{C}^{n+1}$. The face and degeneracy maps, and the cyclic action, are defined as follows:
	\begin{align*}
	d_i(f_0\otimes\dots\otimes f_n)=&
	\begin{cases}
	\; f_0\otimes f_1\otimes\dots\otimes f_i\circ f_{i+1}\otimes\dots\otimes f_n & \quad \text{ if } 0\leq i\leq n-1, \\
	\; (-1)^{n+\sigma} f_n\circ f_0\otimes f_1\otimes\dots\otimes f_{n-1} & \quad \text{ if } i=n,
	\end{cases}
	\\
	s_i(f_0\otimes\dots\otimes f_n)=&
	\begin{cases}
	\; f_0\otimes f_1\otimes\dots\otimes f_i\otimes \mathrm{id}_{C_{i+1}}\otimes f_{i+1}\otimes \dots\otimes f_n & \quad \text{ if } 0\leq i\leq n-1, \\
	\;  f_0\otimes f_1\otimes\dots\otimes f_n\otimes \mathrm{id}_{C_0} & \quad \text{ if } i=n,
	\end{cases}
	\\
	t(f_0\otimes \dots\otimes f_n)=&\quad(-1)^{n+\sigma}(f_n\otimes f_0\otimes\dots\otimes f_{n-1}),
	\end{align*}
	where $\sigma=(\deg f_n)(\deg f_{n-1}+\dots+\deg f_0)$.
\end{df}
We get a covariant  functor from the category of small dg-categories over $k$ to the category of cyclic $k$-modules.  To every cyclic $k$-module $M$, we can further associate a mixed complex by letting $b\colon M_n\to M_{n-1}$ be the alternating sum 
\begin{equation}\label{ksugalaska}
b\coloneqq\sum_{i=0}^n(-1)^i d_i
\end{equation}
of face maps, and by defining the
cochain map $B\colon M_n\to M_{n+1}$ as the composition 
\begin{equation}\label{ksugalaskaB}
B\coloneqq (-1)^{n+1}(1-t_{n+1})sN.\end{equation} Here, $s$ denotes the extra degeneracy 
$s=(-1)^{n+1}t_{n+1}s_n\colon M_n\to M_{n+1}$  and $N\coloneqq\sum_{i=0}^n t_{n+1}^i$.

\begin{rem}\label{mixtodg}
	Let $M$ be a cyclic module. Then, the triple $(M,b,B)$, where $b$ and $B$ are the differentials \eqref{ksugalaska} and \eqref{ksugalaskaB}, respectively,  is a mixed complex. Morphisms of cyclic modules commute with the  face and the  degeneracy maps and with the cyclic operators;  they yield in this way morphisms of mixed complexes and a functor from the category of  cyclic modules to the category of mixed complexes.
\end{rem}

\begin{df}\label{mixdg} \cite[Def.\,1.3]{keller}
	We denote by
	\begin{equation*}
	\fMix\colon \dgcat_{k}\to \Mix
	\end{equation*}
	the functor from the category of small dg-categories over $k$ to the category of mixed complexes 
	defined as composition of the additive cyclic nerve functor of Definition \ref{acn} with the functor of Remark~\ref{mixtodg}.
\end{df}

Thanks to the work of Keller, we know that this functor enjoys many useful properties, among others agreement, additivity and localization \cite{keller}.  As we  work in the context of $\infty$-categories, we will spell them out in this language. 

From now on we assume that $k$ is a field. 
The $\infty$-category of small  dg-categories
$
\dgcat_{k,\infty}\coloneqq \Nerve(\dgcat_{k})[W_{\text{Morita}}^{-1}]
$
is the localization at the  class $W_{\text{Morita}}$  of Morita equivalences.
By \cite[Thm.~1.5]{keller}, the functor $\fMix$ of Definition \ref{mixdg} sends Morita equivalences of dg-categories 
to quasi-isomorphisms of mixed complexes 
and  descends to a functor 
\begin{equation*}\label{functorMix}
\begin{tikzcd}
\dgcat_{k} \arrow{d}{\loc} \arrow{r}{\fMix} & \Mix\arrow{d}{\loc}\\  \dgcat_{k,\infty}\arrow{r}{\fMix} &  \Mix_{\infty}% \arrow{r}{N_{\text{dg}}} & \LMod_\infty
\end{tikzcd}
\end{equation*}
between  the localizations.
Keller's Localization Theorem \cite[Thm.~1.5]{keller} can then be summarized as follows:

\begin{thm}\cite[Thm.\,1.5]{keller}\label{mainkeller}
	Let $k$ be a field.	The  functor  $ \fMix\colon \mathbf{dgcat}_k\to  \Mix_{\infty}$ satisfies the following:
	\begin{enumerate}
		\item\label{kjjdkzlsg} it sends equivalences of small dg-categories to equivalences of mixed complexes;
		\item it commutes with filtered colimits;
		\item it sends short exact sequences $\A\to\B\to\CC$ of dg-categories to cofiber sequences of $\Mix_{\infty}$.
	\end{enumerate}
	Moreover, if $A$ is a $k$-algebra,  there is an equivalence of mixed complexes
	\[
	\fMix(A)\to \fMix(\mathrm{proj}\,A)\] where $\mathrm{proj}\,A$ is the additive category of finitely generated projective modules.
\end{thm}

Observe that the functor $\fMix\circ\loc$ preserves filtered colimits. 
By Proposition \ref{mixstconpl}, the $\infty$-category $\Mix_{\infty}$ is stable and cocomplete and cofiber sequences of $\Mix_{\infty}$ \cite[Def.~1.1.1.6]{HA} are detected in its homotopy category, \emph{i.e.,} in $\D(\Lambda)$.   
We observe here that Keller's theorem holds in a more general setting (for more general rings and for exact categories). However,  we only need these properties in the context of additive categories (over a field). Moreover, in this context, Keller's functor $\fMix$ is equivalent to the cyclic homology functor constructed by McCarthy \cite{mccarthy} (see also \cite[Lemma 3.4.4 \& Rmk.~3.4.5]{mythesis}).

\section{Equivariant coarse Hochschild and cyclic homology}\label{ahkbfkjabljavö}

For a fixed base field $k$ and group $G$, we define equivariant coarse Hochschild
$
\XHH_{k}^G
$
and cyclic homology 
$
\XHC_k^G
$
versions of the classical Hochschild and cyclic homology of $k$-algebras. This is achieved  by first studying an intermediate equivariant coarse homology theory $\XC_{k}^G$ with values in the $\infty$-category of mixed complexes. In the definition of $\XC_{k}^G$, we employ  Keller's functor $\fMix\colon \mathbf{dgcat}_{k}\to \Mix$ in the context of $k$-linear categories, and we apply it to the category $V_k^G(X)$ of controlled objects; we  then define Hochschild/cyclic homology of a bornological coarse space $X$ as Hochschild/cyclic homology of $\mathrm{Mix}(V_k^G(X))$. We conclude the section with some properties of these homology theories and of the associated assembly maps.

\subsection{The equivariant coarse homology theory $\XC_{k}^G$}

Let $k$ be a field,   $\mathbf{Cat}_{k}$ the category of small $k$-linear categories,  $V_{k}^{G}\colon G\BC\to\Cat_k$  the functor of Remark   \ref{ljsahaòlsghgjòal}, let $\fMix\colon \bdgcat_{k}\to \Mix$ be the functor of Definition~\ref{mixdg}, $\iota\colon \Cat_{k}\to\dgcat_{k}$ the functor sending a $k$-linear category to its associated dg-category and $\loc$  the localization functor $\loc\colon \Mix\to\Mix_\infty$.

\begin{df}\label{defXHH}
	We denote by $\XC_k^G$ the following  functor
	\begin{equation*}
	\begin{tikzcd}
	\XC_k^\G\colon \G\BC \arrow{r}{V_{k}^{\G}} & \mathbf{Cat}_{k}\arrow{r}{\iota}&\dgcat_{k} \arrow{r}{\fMix} & \Mix\arrow{r}{\loc}&\Mix_\infty  
	\end{tikzcd}
	\end{equation*}
	from the category of $G$-bornological coarse spaces to the $\infty$-category of mixed complexes.
\end{df}

The proof that the functor $\XC_k^G$ satisfies the axioms of Definition \ref{hcht} describing an equivariant coarse homology theory follows the ideas of coarse algebraic $K$-homology \cite[Sec.~8]{bunkeeq} and does not require  assumptions on $G$ or $k$. For every $G$-set $X$ and additive category with $G$-action $\bA$, the category $V_\bA^G(X_{\text{min,max}}\otimes \Gcan)$ is  equivalent to the  additive category $\bA *_G X$ \cite[Def.~2.1]{BARTELS2007337}  by \cite[Def.~8.21 \& Prop.~8.24]{bunkeeq}. As a consequence, $\XC_{k}^G(X_{\text{min,max}}\otimes \Gcan)\simeq \mathrm{Mix}(\bA*_G X)$, and the functor $\XC_{k}^G$, when applied to such spaces, can be described as the mixed complex of a suitable additive category.  In the case of the group $G$, this result says that the category of controlled objects  $V_k^G(\Gcan)$ is equivalent to the category  of finitely generated free $k[G]$-modules (see Proposition \ref{iso}), which, together with Proposition \ref{pointHC}, further justifies the use of Definition \ref{defXHH} in the construction of a  cyclic homology theory for $G$-bornological coarse spaces.

The main result of the section is the following theorem:
\begin{thm}\label{echt}
	The  functor \[
	\begin{tikzcd}	
	\XC_k^\G\colon \G\BC \arrow{r}& \Mix_\infty 
	\end{tikzcd}
	\]
	is a $\G$-equivariant $\Mix_\infty$-valued coarse homology theory.
\end{thm}

\begin{proof}
	The category $\Mix_{\infty}$ is stable and cocomplete by Proposition \ref{mixstconpl}. We prove below that the functor ${\XC}_{k}^G$ satisfies  coarse invariance (see Proposition \ref{prop1}),  vanishing on flasque spaces (see Proposition~\ref{prop2}),  u-continuity (see Proposition \ref{prop3}) and 
	coarse excision (see Theorem \ref{excision}), 
	\emph{i.e.,} the axioms describing an equivariant coarse homology theory. 
\end{proof}

\begin{prop}\label{prop1}
	The functor ${\XC}_{k}^G \colon G\BC\to\Mix_\infty$ satisfies coarse invariance.
\end{prop}

\begin{proof}
	If $f\colon X\to Y$ is a coarse equivalence of $\G$-bornological coarse spaces, then it induces a natural equivalence 
	$f_*\colon V_k^\G(X)\to V_k^\G(Y)$ by Lemma \ref{coarseinvariance}. 
	Keller's functor  $\fMix$ sends equivalences of dg-categories to equivalences of mixed complexes by Theorem~\ref{mainkeller}~\eqref{kjjdkzlsg}. 
	Hence, the functor $f_{*}$ induces an equivalence ${ \XC}_k^G(X)\xrightarrow{\sim}{ \XC}_k^G(Y)$  in $\Mix_{\infty}$, \emph{i.e., } the functor ${\XC}_{k}^G$ is coarse invariant. 
\end{proof}

Recall the definition of flasque spaces in Definition \ref{flasque}. 

\begin{prop}\label{prop2}
	The functor ${\XC}_{k}^G \colon G\BC\to\Mix_\infty$  vanishes on flasque spaces.
\end{prop}

\begin{proof}
	By Lemma \ref{flasquelemma}, the  category $V_k^\G(X)$ is a flasque category, hence there exists an endofunctor 
	$S\colon V_k^\G(X)\to V_k^\G(X)$  such that $\mathrm{id}_{V_k^\G(X)}\oplus S\cong S$.  By  \cite[Thm.~2.3.11]{schlichting} (see also \cite[Thm.~3.3.5]{mythesis}), the morphisms  $\fMix(\mathrm{id}) \oplus \fMix(S)$ and $\fMix(\mathrm{id} \oplus S)\cong \fMix(S)$ are equivalent in $\Mix_{\infty}
	$. This means that  the morphism  $${\XC}_{k}^G(\mathrm{id})\colon {\XC}_{k}^G(X)\to {\XC}_{k}^G(X)$$ 
	is equivalent to the $0$-morphism and that  ${\XC}_{k}^G(X)\simeq 0$.
\end{proof}

\begin{prop}\label{prop3}
	The functor ${\XC}_{k}^G \colon G\BC\to\Mix_\infty$  is  $u$-continuous.
\end{prop}

\begin{proof}
	Let $X$  be a $\G$-bornological coarse space, and let $\CC^\G$ be the poset of $\G$-invariant controlled sets. By Remark \ref{colimits}, there is an equivalence  $V_k^\G(X)\simeq \colim_{\U\in\CC^\G}V_k^\G(X_\U) $ of $k$-linear  categories, hence of dg-categories. The functor $\fMix$ commutes with filtered colimits,
	and we get the equivalence $${\XC}_{k}^G(X)\simeq \colim_{\U\in\CC^\G}{\XC}_{k}^G(X_\U)$$ in $\Mix_{\infty}$,  which shows that the functor ${\XC}_{k}^G$ is u-continuous.
\end{proof}

\begin{thm}\label{excision}
	The functor ${\XC}_{k}^G \colon G\BC\to\Mix_\infty$ satisfies coarse excision.
\end{thm}

Before giving the  proof of this theorem we first need some more terminology.

\begin{df}\cite{Kasprowski}
	A full additive subcategory $\A$ of an additive category $\cU$ is a \emph{Karoubi-filtration} if every diagram
	$
	X\to Y\to Z
	$
	in $\cU$, with $X,Z\in \A$, admits an extension
	\[
	\begin{tikzcd}
	X \arrow{d} \arrow{r} & Y \arrow{r} \arrow{d}{\cong} & Z \\
	A \arrow{r}{i} & A\oplus A' \arrow{r}{p} & A\arrow{u} 
	\end{tikzcd}
	\]
	for some object $A\in \A$.
\end{df}
By  \cite[Lemma 5.6]{Kasprowski}, this definition is equivalent to the classical one 
\cite{Karoubi1970,cp}.
If  $\A$ is a Karoubi-filtration of $\cU$, we can construct a quotient category $\cU/\A$. Its objects are the objects of $\cU$, and the morphisms sets are defined as follows:
\[
\Hom_{\cU/\A}(U,V)\coloneqq \Hom_\cU(U,V)/{\sim}
\]
where the relation identifies pairs of maps $U\to V$ whose difference factors through an object of $\A$.

Let $X$ be a $\G$-bornological coarse space and 
let $\Y=(Y_i)_{i\in I}$ be an equivariant big family on $X$ (see Definition \ref{complementary}). The bornological coarse space $Y_i$ is a subspace of $X$ with the induced bornology and coarse structure. The inclusion $Y_i\hookrightarrow X$ induces a functor $V_k^\G(Y_i)\to V_k^\G(X)$ which is injective on objects. The categories $V_k^\G(Y_i)$
and 
$
V_k^\G(\Y)\coloneqq\colim_{i\in I} V_k^\G(Y_i)
$
are  full subcategories of $V_k^\G(X)$.

\begin{lemma}\cite[Lemma 8.14]{bunkeeq}\label{karoubincl}
	Let $\Y$ be an equivariant big family on the $G$-bornological coarse space $X$.
	Then, the full additive subcategory $V_k^\G(\Y)$ of  $V_k^\G(X)$ is a Karoubi filtration.
\end{lemma}

Let $X$ be a $\G$-bornological coarse space, and let $(Z,\Y)$ be an equivariant  complementary pair. Consider the functor
\begin{equation}\label{functora}
a\colon V_k^\G(Z)/V_k^\G(Z\cap\Y)\to V_k^\G(X)/V_k^\G(\Y)
\end{equation}
induced by the inclusion  $i\colon Z\to X$; 
on objects, it coincides with $i_*\colon  V_k^\G(Z) \to V_k^\G(X)$, but on morphisms it sends the equivalence class $[A]$ of $A$ to the equivalence class $[i_*(A)]$ of $i_*(A)$.

\begin{lemma}\cite[Prop.\,8.15]{bunkeeq}\label{equiv}
	The functor $a$ in \eqref{functora}
	is  an equivalence of categories.
\end{lemma}

\begin{proof}[Proof of Theorem \emph{\ref{excision}}]
	Let $X$ be a $\G$-bornological coarse space, and let $(Z,\Y)$ be an equivariant  complementary pair on $X$.
	By Lemma \ref{karoubincl},  $V_k^\G(Z\cap\Y)\subseteq V_k^\G(Z)$ and $V_k^\G(\Y)\subseteq V_k^\G(X)$ are Karoubi filtrations and yield the
	following sequences of $k$-linear  categories:
	\[
	V_k^\G(Z\cap\Y)\to V_k^\G(Z)\to V_k^\G(Z)/V_k^\G(Z\cap\Y)
	\]
	and
	\[
	V_k^\G(\Y)\to V_k^\G(X)\to V_k^\G(X)/V_k^\G(X\cap\Y).
	\]
	By  \cite[Ex.~1.8, Prop.~2.6]{Schlichting04} (see also \cite[Rem.~3.3.12]{mythesis}), Karoubi filtrations  induce short exact sequences of dg-categories.
	Hence, 
	Theorem \ref{mainkeller}  gives  cofiber sequences of mixed complexes.
	The inclusion $Z\hookrightarrow X$ induces a commutative diagram (where the rows are the obtained cofiber sequences)
	\[
	\begin{tikzcd}
	\fMix(V_k^\G(Z\cap\Y)) \arrow{d} \arrow{r} & \fMix(V_k^\G(Z)) \arrow{d} \arrow{r} & \fMix(V_k^\G(Z)/V_k^\G(Z\cap\Y))\arrow[dashrightarrow]{d}{a_*} \\
	\fMix(V_k^\G(\Y)) \arrow{r} & \fMix(V_k^\G(X)) \arrow{r} & \fMix(V_k^\G(X)/V_k^\G(X\cap\Y));
	\end{tikzcd}
	\]
	here $a_*$ is the map induced by $a\colon V_k^\G(Z)/V_k^\G(Z\cap\Y)\to V_k^\G(X)/V_k^\G(\Y)$ in \eqref{functora}. By Lemma~\ref{equiv}, the functor $a$   yields an equivalence of categories, hence of mixed complexes and the left square  is a co-Cartesian square in $\Mix_{\infty}$.
	
	In order to conclude the proof, we recall that ${\XC}_k^\G(\Y)$ is defined as the filtered colimit ${\XC}_k^\G(\Y)=\colim_i {\XC}_k^\G(Y_i)$  and that 
	$
	V_k^\G(\Y)\coloneqq \colim_{i\in I} V_k^\G(Y_i).
	$
	The functor $\fMix$ commutes with filtered  colimits of dg-categories. Hence we have the   equivalence $\fMix(V_k^\G(\Y))= \fMix(\colim_i V_k^\G(Y_i))\simeq \colim_i \fMix(V_k^\G(Y_i))$, and the same holds for $Z\cap\Y$.   
	By using these identifications, we obtain the co-Cartesian square in $\Mix_{\infty}$
	\[
	\begin{tikzcd}
	{\XC}_{k}^G(Z\cap\Y) \arrow{d} \arrow{r} & {\XC}_{k}^G(Z) \arrow{d}   \\
	{\XC}_{k}^G(\Y) \arrow{r} & {\XC}_{k}^G(X).
	\end{tikzcd}
	\]
	This	means that  ${\XC}_{k}^G$ satisfies coarse excision.
\end{proof}

\subsection{Coarse Hochschild and cyclic homology}

The  functors $\mathrm{forget}\colon\Mix\to\bCh$ in \eqref{forgetf}, sending a mixed complex to the underlying chain complex, and $\Tot(\B-)\colon \Mix\to\bCh$ in~\eqref{totfunct},  sending a mixed complex to the total complex of its associated bicomplex, send quasi-isomorphisms of mixed complexes to quasi-isomorphisms of chain complexes. Hence they descend to functors between the localizations. 

\begin{df}\label{eqcHH}
	Let $k$ be a field, $G$ a group and $\bCh_{\infty}$ the $\infty$-category of chain complexes.
	The \emph{$\G$-equivariant coarse Hochschild homology} $\XHH_k^\G$ (with $k$-coefficients) is the $\G$-equivariant $\bCh_\infty$-valued coarse homology theory 
	\[
	\begin{tikzcd}
	\XHH_k^\G\colon \G\BC \arrow{r}{\XC_{k}^{\G}} &\Mix_\infty \arrow{r}{\mbox{forget}} & \bCh_\infty
	\end{tikzcd}
	\]
	defined as  composition of the functor 
	$\XC_{k}^G$ of Definition \ref{defXHH} and of the functor $\mathrm{forget}$ in~\eqref{forgetf}. The composition 
	\[
	\begin{tikzcd}
	\XHC_k^\G\colon \G\BC \arrow{r}{\XC_{k}^{\G}} &\Mix_\infty \arrow{r}{\Tot(\B-)} & \bCh_\infty
	\end{tikzcd}
	\]
	involving composition with the functor $\Tot(\B-)$ \eqref{totfunct} is  \emph{$\G$-equivariant coarse cyclic homology}.
\end{df}

The definitions are justified by the following:

\begin{thm}\label{gajhljaghjlgalaöghga}
	The functors
	\[
	\XHH_k^\G\colon\G\BC\to\mathbf{Ch}_\infty \quad \text{ and } \quad 	\XHC_k^\G\colon\G\BC\to\mathbf{Ch}_\infty
	\]	are $\G$-equivariant $\bCh_{\infty}$-valued coarse homology theories. 
\end{thm}

\begin{proof}
	By Theorem  \ref{echt}, the functor $\XC^\G_k\colon G\BC\to \Mix_{\infty}$ is an equivariant coarse homology theory. 
	The functors $\mathrm{forget}\colon\Mix_\infty\to \bCh_\infty$   and the functor $\Tot(\B-)\colon \Mix_\infty\to \bCh_\infty$ commute with filtered colimits and send cofiber sequences to cofiber sequences. Hence the two compositions with $\XC^\G_k$ satisfy coarse invariance, coarse excision, $u$-continuity, and vanishing on flasques.
\end{proof}

The category of mixed complexes has a natural symmetric monoidal structure  induced by tensor products between the underlying chain complexes \cite{kassel}. As tensor products of mixed complexes preserve equivalences, $k$ being a field,  we get a symmetric monoidal $\infty$-category  $\Mix_{\infty}^\o\coloneqq \Nerve(\Mix^\o)[W_{\mathrm{mix}}^{\o,-1}]\to\Nerve(\Fin_*)$ (with monoidal structure induced by the monoidal structure on $\Mix$ by \cite[Prop.~3.2.2]{hinich}). 
By \cite[Thm.~2.4]{kassel},  the functor $\fMix$  has a lax symmetric monoidal refinement 
(see also \cite{MR2922389}). This implies that coarse Hochschild and cyclic homologies are lax symmetric monoidal functors:

\begin{prop}\label{XCCsymmmon}
	The functors  $\XHH_{k}^\G$ and $\XHC_{k}^\G$  admit lax symmetric monoidal refinements:
	\[
	\XHH_{k}^{\G,\o}\colon \Nerve(\G\BC^\o)\to\bCh_{\infty}^\o
	\]
	and 
	\[
	\XHC_{k}^{\G,\o}\colon \Nerve(\G\BC^\o)\to\bCh_{\infty}^\o,
	\]
	where $\bCh_{\infty}^\o$ is the $\infty$-category of chain complexes with its standard symmetric monoidal structure.
\end{prop}

\begin{proof}
	By \cite[Thm.~3.26]{symm} and \cite[Thm.~2.4]{kassel}, the functor $\XC_{k}^G$ of Definition \ref{defXHH}  admits a lax symmetric monoidal refinement
	\[
	\XC_{k}^{\G,\o}\colon \Nerve(\G\BC^\o)\to\Mix_{\infty}^\o.
	\]
	As the functors $\mathrm{forget}$ in \eqref{forgetf} and $\Tot(\B-)$ in \eqref{totfunct}
	are lax symmetric monoidal,
	coarse Hochschild and cyclic homology admit lax symmetric monoidal refinements as well.
\end{proof}

\subsection{Comparison results and assembly maps}\label{sectioncomparson}

In this subsection, 
we compare equivariant coarse Hochschild homology with the classical version of Hochschild homology for $k$-algebras. Furthermore, we show that the forget-control map for coarse Hochschild homology is equivalent to the associated generalized assembly map. 

\begin{notation}
	Let $A$ be a $k$-algebra.
	We denote by $$C^{\HH}_{*}(A;k)\quad \text{ and  } \quad C^{\HC}_{*}(A;k)$$ the chain complexes computing the  Hochschild and cyclic homology  of the mixed complex $\fMix(A)$ associated to the cyclic object $Z_*(A)$ associated to $A$ \cite{goodwillie, loday}.
\end{notation}

Let $\{*\}$ be the one-point bornological coarse space, endowed with a trivial $G$-action.

\begin{prop}\label{pointHC}
	There are equivalences of chain complexes $$\mathcal{X}\mathrm{HH}_{k}(*)\simeq C_{*}^{\mathrm{HH}}(k;k)\quad \text{ and } \quad\mathcal{X}\mathrm{HC}_{k}(*)\simeq C_{*}^{\mathrm{HC}}(k;k)$$ between the coarse Hochschild \textup{(}cyclic\textup{)} homology of the point and the classical Hochschild \textup{(}cyclic\textup{)} homology  of $k$.
\end{prop}

\begin{proof}
	By Theorem \ref{mainkeller}, the mixed complex $\fMix(A)$ associated to a $k$-algebra $A$ is equivalent to the mixed complex associated to the $k$-linear category of 
	finitely generated projective $A$-modules. 
	When $X$ is a point endowed with a trivial $\G$-action and $k$ is a field, the $k$-linear category $V_k(X)$ is isomorphic to the category $\mathbf{Vect}_k^{\mathrm{f.d.}}$ of finite-dimensional
	$k$-vector spaces, \emph{i.e.,}  $\fMix(V_k(\{*\}))\simeq \fMix\big(\mathbf{Vect}_k^{\mathrm{f.d.}}\big)\simeq \fMix(k)$.
\end{proof}

Let $G$ be a group. By Example \ref{equivex},  there is a canonical $G$-bornological coarse space  $\G_{\mathrm{can,min}}=(G,\CC_{\text{can}},\B_{\min} )$  associated to it.

\begin{prop}\label{iso}
	There are equivalences of chain complexes:
	\[
	\XHH_k^\G(\Gcan)\simeq C_{*}^{\HH}(k[\G];k)
	\]
	and
	\[
	\XHC_k^\G(\Gcan)\simeq C_{*}^{\HC}(k[\G];k)
	\] 
	between the $\G$-equivariant coarse Hochschild and cyclic homologies of $\Gcan$ and the classical Hochschild and cyclic homologies of the group algebra $k[\G]$.
\end{prop}

\begin{proof}
	The category $V_k^\G(\G_{\mathrm{can,min}})$ of $\G$-equivariant $\G_{\mathrm{can,min}}$-controlled finite-dimensional $k$-vector spaces  is equivalent to
	the category $\mathbf{Mod}^{\mathrm{fg,free}}(k[\G])$ of finitely generated free $k [\G]$-modules  \cite[Prop.~8.24]{bunkeeq}. By Theorem \ref{mainkeller}, Keller's mixed complex $\fMix(\mathbf{Mod}^{\mathrm{fg,free}}(k[\G]))$ of the category of finitely generated free $k[G]$-modules is  equivalent to the mixed complex associated to the category $\mathbf{Mod}^{\mathrm{fg,proj}}(k[\G])$ of finitely generated projective modules (because they are Morita equivalent dg-categories). Therefore, the result follows from the chain of equivalences of mixed complexes
	\[
	\fMix(V_k^\G(\G))\simeq \fMix(\mathbf{Mod}^{\mathrm{fg,free}}(k[\G]))\simeq \fMix(\mathbf{Mod}^{\mathrm{fg,proj}}(k[\G]))\simeq \fMix(k [\G]),
	\]
	where the last equivalence is again true  by Theorem \ref{mainkeller}. 
\end{proof}

Let $X$ be a $G$-set and let $X_{\mathrm{min,max}}$ denote the $G$-bornological coarse space with minimal coarse structure and maximal bornology.

\begin{rem}\label{calculation}
	Let $H$ be a subgroup of $G$; then, by \cite[Prop.~8.24]{bunkeeq} we get an equivalence of chain complexes:
	\[
	\XHH_{k}^\G((\G/ H)_{\mathrm{min,max}}\o\Gcan)\simeq C_{*}^{\HH}(k[H];k);
	\]
	the same holds for equivariant coarse cyclic homology.
\end{rem}

One of the main applications of coarse homotopy theory is the study of assembly maps. We conclude this subsection with  a comparison result between the forget-control maps  for equivariant coarse Hochschild  and cyclic  homology and the associated assembly maps. 
Recall the definitions of   the cone functor $\mathcal{O}^\infty_{hlg}$ \cite[Def.~10.10]{bunkeeq}, of the forget-control map $\beta$ \cite[Def.~11.10]{bunkeeq} and of the coarse assembly map $\alpha$ \cite[Def.~10.24]{bunkeeq}. 
By  \cite[Thm.~11.16]{bunkeeq}, the forget-control map for a $G$-equivariant coarse  homology theory $E$ can be compared with the classical assembly map for the associated $G$-equivariant homology theory $E\circ\mathcal{O}^\infty_{hlg}\colon G\mathbf{Top} \to \bC$.

By applying the Eilenberg-MacLane correspondence \eqref{EMcorr}, we can assume that the equivariant coarse homology theories  $\XHH_{k}^\G$ and $\XHC_{k}^\G$ are equivariant spectra-valued coarse homology theories. 

\begin{df}
	Let $\mathbf{HH}_k^G\coloneqq \mathcal{EM}\circ \XHH_k^G\circ\cO^\infty_{\text{hlg}}\colon G\mathbf{Top}\to \Sp$ be the 
	$\G$-equivariant homology theory associated to equivariant coarse Hochschild homology.
\end{df}

Let $\Fin$ be the family of finite subgroups of $\G$. The following is a consequence of \cite[Thm.~11.16]{bunkeeq} (see also \cite[Prop.~4.2.7]{mythesis}):

\begin{prop}\label{assembly}
	The forget-control map $\beta_{\Gcan,G_\mathrm{{max,max}}}$ for $\XHH_{k}^\G$ is equivalent to the assembly map $\alpha_{E_{\mathbf{Fin}}G,\Gcan}$ for the $G$-homology theory $\mathbf{HH}_k^G$. 
\end{prop}

Furthermore, the assembly map $\alpha_{E_{\mathbf{Fin}}G,\Gcan}$ for the $G$-homology theory $\mathbf{HH}_k^G$ (hence, the forget-control map $\beta_{\Gcan,G_\mathrm{{max,max}}}$ for $\XHH_{k}^\G$) is split injective by \cite[Thm.~1.7]{lr}.

\section{From coarse algebraic $K$-theory to coarse ordinary homology}\label{hasglkfhagsakö}
In this section we define a natural transformation $$\Phi_{\XHH^\G_k}\colon \XHH^\G_k\longrightarrow \XH^\G$$ 
from equivariant coarse Hochschild homology  $\XHH_k^G$ to the $\bCh_{\infty}$-valued equivariant coarse ordinary homology $\XH^\G$ and, analogously,  a natural transformation $\Phi_{\XHC^\G_k}$
from equivariant coarse cyclic homology. By abuse of notation we will denote by $\XHH^G$, $\XHC_{}^G$ and $\XH^G$ both the chain and spectra valued coarse homology theories.

The transformation $\Phi_{\XHH^\G_k}$ is constructed in the following steps:
\begin{itemize}
	\item For every $G$-bornological coarse space $X$, we consider its associated $k$-linear category $V_k^G(X)$ of controlled objects, hence the associated additive cyclic nerve $\CN_{}(V_k^G(X))$. For every tensor element $A_0\o_{}\dots\o A_n$ in the additive cyclic nerve of $V_k^G(X)$ and every  $n+1$ points $x_0,\dots,x_n$ of $X$, we define a  trace-like map, which gives an element of $k$ (see Notation \ref{not1});
	\item by letting $x_0,\dots,x_n$ vary, this  yields a $G$-equivariant locally finite controlled chain on $X$, \emph{i.e.,} an element of $\X C_n^G(X)$ (see  Definition \ref{phi} and Lemma \ref{khfgs,akhfagafskha});
	by letting $A_0\o_{}\dots\o A_n$ vary we get a map $\phi\colon \CN_{*}(V_k^G(X))\to \X C_*^G(X)$ that is a chain map with respect to the differential $d=\sum (-1)^i d_i$ of $\CN_{}(V_k^G(X))$ (see Proposition \ref{diffb});
	\item the additive cyclic nerve $\CN_{}(V_k^G(X))$ yields a mixed complex with the differentials $b$   and  $B$ as in Remark \ref{mixtodg}; the chain map $\phi$ extends to a map of mixed complexes $\tilde\phi$ (see Lemma \ref{kakhbaskölgsöklöga}) and yields a natural transformation of equivariant coarse homology theories   $\Phi_{\XHH^\G_k}\colon \XHH^\G_k\longrightarrow \XH^\G$ (see Theorem \ref{suzfagsazgsafg}).
\end{itemize}

We now proceed with the precise construction. \\

Let $V_k^\G(X)$ be the $k$-linear category of $X$-controlled finite-dimensional $k$-vector spaces of Definition \ref{controbj}.
The  additive cyclic nerve associated to $V_{k}^G(X)$   (see Definition \ref{acn})  is described, in degree $n$, by
\[
\mathrm{CN}_n(V_k^\G(X))=\bigoplus_{((M_0,\rho_0),\dots,(M_n,\rho_n))}\left(\bigotimes_{i=0}^n\Hom((M_{i+1},\rho_{i+1}),(M_i,\rho_i))\right),
\] 
where the index
$i$ runs cyclically in the set  $\{0,\dots, n\}$ and the sum ranges over all the $(n+1)$-tuples $((M_0,\rho_0),\dots,(M_n,\rho_n))$ of objects of  $V_k^\G(X)$.

\begin{rem}
	For every controlled morphism $A_{i}\colon (M_{i+1},\rho_{i+1})\to (M_i,\rho_i)$ (see Definition~\ref{Uequivmorh}) in $\Hom((M_{i+1},\rho_{i+1}),(M_i,\rho_i))$  and for every pair of points $x$ and $y$ of $ X$, there is a well-defined $k$-linear map  $$A_i^{x,y}\colon M_{i+1}(x)\to M_{i}(y)$$ induced by $A_{i}$.
\end{rem}

We use the  following notation:

\begin{notation}\label{not1}
	Let $A_0\o\cdots\o A_n$ be an element of $\bigotimes_{i=0}^n\Hom((M_{i+1},\rho_{i+1}),(M_i,\rho_i))$ 
	and let  $((M_0,\rho_0),\dots,(M_n,\rho_n))$ be an $(n+1)$-tuple of objects of  $V_k^\G(X)$. Let $(x_0,\dots,x_n)$ be a point of $X^{n+1}$.
	The symbol  $$(A_0\circ\cdots\circ A_n)|(x_0,\dots,x_n)$$ 
	denotes the linear operator $(A_0\circ\dots\circ A_n)|(x_0,\dots,x_n) \colon M_{0}(x_{n})\to M_{0}(x_{n})$ defined as the composition
	\begin{align*}
	(A_0\circ\cdots&\circ A_n)|(x_0,\dots,x_n)\coloneqq  \\   &M_0(x_n)\xrightarrow{A_n^{x_n,x_{n-1}}} M_{n}(x_{n-1})\xrightarrow{A_{n-1}^{x_{n-1},x_{n-2}}} \dots\xrightarrow{A_1^{x_1,x_0}} M_{1}(x_0)\xrightarrow{A_0^{x_0,x_n}} M_0(x_n)
	\end{align*}
	of the induced operators $A_i^{x_i,x_{i+1}}\colon M_i(x_i)\to M_{i+1}(x_{i+1})$. It
	is an endomorphism of $M_0(x_n)$, which is a finite-dimensional $k$-vector space.  
\end{notation}

Let $X$ be a $G$-bornological coarse space and let $\X C_n(X)$ be the $k$-vector space generated by the locally finite controlled  $n$-chains on $X$ (see Definition \ref{nchaincontrolled}).

\begin{df}\label{phi}
	We 
	let $\phi_n\colon \mathrm{CN}_n(V_k^\G(X))\to \X C_n(X)$ be the map  defined on elementary tensors  as 
	\begin{equation*}
	\begin{split}
	\phi_n\colon & A_0\o\cdots\o A_n\longmapsto \\ &\sum_{(x_0,\dots,x_n)\in X^{n+1}} \tr\left((A_0\circ\cdots\circ A_n)|(x_0,\dots,x_n)\colon M_0(x_n)\to M_0(x_n)\right)\cdot(x_0,\dots,x_n)
	\end{split}
	\end{equation*}
	and extended linearly.
\end{df}

\begin{lemma}\label{khfgs,akhfagafskha}
	The $n$-chain	
	$
	\phi_n(A_0\o\cdots\o A_n)$ 
	is locally finite and controlled.
\end{lemma}

\begin{proof}
	In order to prove that $\phi_n(A_{0}\o\cdots\o A_{n})$ is locally finite and controlled  we show that its support $\supp(\phi_n(A_{0}\o\cdots\o A_{n}))$ defined in \eqref{suppordhom} is locally finite and that there exists an entourage $\U$ of $X$ such that every $x=(x_0,\dots,x_n)$ in $\supp(\phi_n(A_{0}\o\cdots\o A_{n}))$ is $\U$-controlled. 
	
	We first observe that the operators  $A_i\colon (M_{i+1},\rho_{i+1})\to (M_{i},\rho_i)$ are $U_{i}$-controlled for some entourage $U_{i}$ of $X$. By Definition \ref{Uequivmorh}, $A_i$ is given by a natural transformation of functors $M_{i+1}\to M_{i}\circ U_{i}[-]$ satisfying an equivariance condition.
	For every point $x$ in $X$, 
	$A_{i}$ restricts to a morphism 
	\[
	M_{i+1}(x)\to M_{i}(U_{i}[x])\cong
	\bigoplus_{x^{\prime}\in U_{i}[x]} M_{i}(x^{\prime}),
	\]
	where the direct sum has only finitely many non-zero summands.
	
	Let $K$ be a bounded set of $X$. The set of points $x_n\in K_{}$ for which $M_0(x_n)$ is non-zero is finite (as a consequence of Definition \ref{catcontrobj}). For such a fixed  $x_n$, there are only finitely many points $x_{n-1}\in U_{n}[K_{}]$ such that the corresponding map $A_n^{x_{n},x_{n-1}}\colon M_0(x_n)\to M_n(x_{n-1})$ is non-zero. The set $U_{n}[K_{}]$ is a bounded set of $X$,  the morphism $A_{n-1}\colon M_{n}\to M_{n-1}$ is $U_{n-1}$-controlled 
	and we can repeat the same argument for $A_{n-1}$, hence for  each $A_{i}$. This implies  that  the $n$-chain is locally finite because, for the given bounded set $K$, we have found only finitely many tuples $(x_0,\dots,x_n)$ in the support of $\phi_n(A_0\o_{}\cdots\o A_n)$ that meet $K$.

	The chain is also $U$-controlled, where $U$ is the entourage $\U\coloneqq U_0\circ\dots\circ U_n$ of $X$.
\end{proof}

\begin{rem}
	Let $X$ be a $G$-bornological coarse space.
	Let $(M,\rho)$ be a $\G$-equivariant $X$-controlled finite-dimensional $k$-vector space and let $g$ be an element of the group $\G$. Then, $\rho(g)$ (Definition \ref{controbj}) is a natural isomorphism between the functors $M$ and $gM$. 
	The diagram 
	\[
	\begin{tikzcd}
	M_0(gx_n) \arrow{d}{\cong} \ar[r,"{A_n^{gx_n,gx_{n-1}}}"] &[+15pt] M_{n}(gx_{n-1}) \arrow{d}{\cong}\arrow{r}{}&[+15pt] \dots\arrow{r}{A_0^{gx_0,gx_n}} \arrow[dotted]{d}{\cong} &[+15pt] M_0(gx_n)\arrow{d}{\cong}\\
	gM_0(x_n) \arrow{r}{gA_n^{x_n,x_{n-1}}} & gM_{n}(x_{n-1})\arrow{r}{}& \dots\arrow{r}{gA_0^{x_0,x_n}}  & gM_0(x_n)
	\end{tikzcd}
	\]
	is commutative, for
	$A_0\o\cdots\o A_n$ in 
	$\mathrm{CN}_n(V_k^\G(X))$ with $A_{i}\colon (M_{i+1},\rho_{i+1})\to (M_i,\rho_i)$,
	where the isomorphisms are induced by $\rho_i(g)$. Hence, the
	image of  $\phi_{n}$ is a $\G$-invariant locally finite controlled $n$-chain on $X$.
\end{rem}

Let $\partial_{i}\colon \X C_n^G(X)\to\X C_{n-1}^G(X)$ be the $i$-th differential of the chain complex~$\X C^G(X)$ and let $d_{i}\colon \CN_n(V_k^\G(X))\to\CN_{n-1}(V_k^\G(X))$ be the $i$-th face map of $\CN(V_k^\G(X))$. In the next proposition we consider the chain complex $(\mathrm{CN}_{}(V_k^\G(X)),d)$ underlying the additive cyclic nerve~$\CN(V_k^\G(X))$. 

\begin{prop}\cite[Prop.\,4.3.6]{mythesis}\label{diffb}
	The maps $\phi_n\colon \mathrm{CN}_n(V_k^\G(X))\to\X C_n^G(X)$ of Definition~\textup{\ref{phi}}
	extend to a chain map 
	$
	\phi\colon (\mathrm{CN}_{}(V_k^\G(X)),d)\to(\X C_{}^G(X),\partial).
	$
\end{prop}

\begin{proof}
	The result is a consequence of the additivity of the trace map and of its invariance under cyclic permutations.
\end{proof}

If $M$ is a cyclic module,  as in Remark \ref{mixtodg}, we get a mixed complex. 
The chain complex $\X C^G(X)$ is also a mixed complex with the differential $B=0$. 

\begin{lemma}\label{kakhbaskölgsöklöga}
	The chain map $\phi\colon \mathrm{CN}(V_{k}^\G(X))\to\X C^\G(X)$ of Definition \textup{\ref{phi}} extends to a map 
	$\tilde{\phi}\colon \mathrm{Mix}(V_{k}^\G(X))\to\X C^\G(X)$ 
	that is a morphism of  mixed complexes. 
\end{lemma}

\begin{proof}
	The proof is a simple computation and uses the definition of the operator $B$ \eqref{ksugalaskaB} and that the trace is invariant under cyclic permutations. See also \cite[Lemma 4.3.7]{mythesis}. 
\end{proof}

We can now construct the natural transformation $\Phi_{\XHH_{k}^\G}\colon \XHH_{k}^\G\to\X C^\G_{}$:

\begin{thm}\label{suzfagsazgsafg}
	The map $\phi$ extends to natural transformations
	\begin{equation*}
	\Phi_{\XHH_{k}^\G}\colon \XHH_{k}^\G\to\XH^\G_{}.
	\end{equation*}
	and 
	\[
	\Phi_{\XHC_{k}^\G}\colon \XHC_{k}^\G\to\bigoplus_{n\in \N}\XH_{}^\G
	\]
	of $G$-equivariant $\bCh_{\infty}$-valued coarse homology theories.
\end{thm}

\begin{proof}
	The  map $\phi\colon (\mathrm{CN}_{}(V_k^\G(X)),d)\to(\X C_{}^G(X),\partial)$ of Definition \ref{phi} is a chain map by Proposition \ref{diffb}. Let $f\colon X\to Y$ be a morphism of $G$-equivariant bornological coarse spaces. Consider the induced chain map $\X C^G(f)\colon \X C^G_{}(X)\to \X C^G_{}(Y)$  and the induced functor $f_{*}=V^G_{k}(f)\colon V^\G_k(X) \to V^\G_k(X')$. By functoriality of the additive cyclic nerve, $f_*$ induces  a  morphism $\CN_{}(f_{*})\colon \CN_{*}(V^\G_k(X))\to \CN_{*}(V^\G_k(Y))$ of  cyclic modules (hence, a chain map between the underlying chain complexes as well).
	
	The diagram 
	\[
	\begin{tikzcd}
	\CN_{n}(V^\G_k(X)) \arrow{r}{{\phi_{n}}}\arrow{d}{\CN_{}(f_{*})}& \X C^G_{n}(X) \arrow{d}{\X C^G(f)}\\
	\CN_{n}(V^\G_k(Y)) \arrow{r}{{\phi_{n}}} & \X C^G_{n}(Y)
	\end{tikzcd}
	\]
	is commutative. 
	The map $\phi$ extends to the associated mixed complexes by Lemma \ref{kakhbaskölgsöklöga} and this extension preserves the commutative  diagram (of associated mixed complexes). After localization and application of the forgetful functor (recall the definition of equivariant coarse Hochschild homology in terms of $\XC_{k}^G$, Definition \ref{eqcHH}), the map $\phi$ yields  a natural transformation of equivariant coarse homology theories $\Phi_{\XHH_{k}^\G}\colon \XHH_{k}^\G\to\XH^\G_{}.$
	
	To every mixed complex $C$, we  associate the chain complex $\Tot(\B C)$ \eqref{dcc} defined by $\Tot_n (\B C)=\bigoplus_{i\geq 0} C_{n-2i} $ with differential $d(c_n,c_{n-2},\dots)=(bc_n + Bc_{n-2},\dots)$. By Lemma~\ref{kakhbaskölgsöklöga},  we conclude that the map $\phi$ extends to a chain map on the total complex as well, and to a natural transformation of coarse homology theories
	\[
	\Phi_{\XHC_{k}^\G}\colon \XHC_{k}^\G\to\bigoplus_{n\in \N}\XH_{}^\G.
	\]
	Here the sum is indexed by the natural numbers because the (mixed complex associated to) the additive cyclic nerve of $V_{k}^G(X)$ is positively graded.
\end{proof}

The following result implies that the transformation $\Phi_{\XHH_{k}}\colon \XHH_{k}\to\XH_{}$ is non-trivial:

\begin{prop}\label{kajsbkbvk}
	If $X$ is the one-point space $\{*\}$, then the transformation $$\Phi_{\XHH_{k}}\colon \XHH_{k}(*)\to\XH_{}(*)$$ induces an equivalence of chain complexes.
\end{prop}

\begin{proof}
	Let $c\colon \{*\}^{n+1}\to k$ be an $n$-chain  in $\X C_n(*)$; we identify this chain with the element $c\in k$ that is its image. Let $\iota_n\colon \X C_n(*)\to \CN_n(V_k(*)))$ be the map sending $c$ to the element $(\cdot c)\o (\cdot 1_{k})\o\cdots\o (\cdot 1_{k})$. This extends to a chain map that gives a section of the trace map, \emph{i.e.,} $\phi\circ \iota=\id$. 
	
	As  coarse Hochschild homology and coarse ordinary homology of the point are both isomorphic to the Hochschild homology of the ground field $k$ (by Example \ref{coarseordinarypoint} and  Proposition \ref{pointHC}), we get equivalences of chain complexes
	\[
	\XHH_{k}(*)\simeq C_{*}^{\HH}(k)\simeq \X C_k(*).
	\]
	By using these equivalences and the section $\phi\circ \iota=\id$, we obtain that, when $X$ is the one-point space, the transformation $\Phi_{\XHH_{k}}$ induces an equivalence of chain complexes.
\end{proof}

By applying the Eilenberg-MacLane correspondence   $\mathcal{EM}$
\eqref{EMcorr}, we now assume  that  equivariant coarse Hochschild and cyclic homology take values in the $\infty$-category $\Sp$ of spectra. The classical trace map constructed by McCarthy \cite[Sec.~4.4]{mccarthy} extends to a transformation  from equivariant coarse algebraic $K$-homology to equivariant coarse Hochschild homology:

\begin{prop}\cite[Prop.\,4.4.1]{mythesis}\label{compkth}
	There are natural transformations $$\mathrm{K}\mathcal{X}_{k}^\G\to\XHH_{k}^\G\quad \text{ and } \quad \mathrm{K}\mathcal{X}_{k}^\G\to\XHC_{k}^\G$$ induced by the Dennis trace maps from algebraic K-theory to Hochschild  homology.
\end{prop}

In particular,  when $X$ is the $G$-bornological coarse space $\Gcan$,  the induced map $$K\X_k^G(\Gcan)\to \XHH_k^G(\Gcan)$$ is the classical Dennis trace map $K(k[G])\to \HH(k[G])$ by McCarthy's agreement result  \cite[Sec.~4.5]{mccarthy}, by \cite[Prop.~8.24]{bunkeeq} and by Proposition \ref{iso}.

Composing the transformations of Proposition \ref{compkth} and of Theorem \ref{suzfagsazgsafg} we get the  
natural transformation $$K\X_k^G\to\XHH_k^G\xrightarrow{ \Phi_{\XHH^\G_k}} \XH_k^G$$ from equivariant coarse algebraic $K$-homology to equivariant coarse ordinary homology. When $X$ is the $G$-bornological coarse space $\Gcan$, we get a map $K(k[G])\to H(G;k)$ from the algebraic $K$-theory of the group ring $k[G]$ to the ordinary homology of $G$ with $k$-coefficients. We believe that further investigations of this transformation can be useful to detect  coarse $K$-theory classes.

\subsection*{Acknowledgements}
This work formed part of the author’s PhD thesis at Regensburg University. It is a pleasure to again acknowledge  Ulrich Bunke,  this work would not exist without him. The author also thanks
Clara Löh,  Denis-Charles Cisinski and Alexander Engel for helpful discussions, and the anonymous referees for constructive comments and recommendations. The author has been supported by the DFG Research Training Group GRK 1692 ``Curvature, Cycles, and Cohomology” and by  the DFG SFB 1085 ``Higher Invariants”.

\bibliographystyle{alpha}
\bibliography{biblio}

\end{document}